\documentclass{amsart}
\usepackage{dino}
\newcommand{\T}{{\mathcal{T}}}
\newcommand{\SR}{SR}
\newcommand{\pSR}{pSR}
\newcommand{\Q}{\mathbb{Q}}
\newcommand{\Z}{\mathbb{Z}}
\newcommand{\K}{\mathcal{K}}
\newcommand{\N}{\mathcal{N}}
\renewcommand{\L}{\mathcal{L}}
\providecommand{\Ginf}[1]{\Gamma^{\mathrm{in}}_{#1}}
\title{Scott sentence complexities of linear orderings}
\author{David Gonzalez}

\author{Dino Rossegger}

\address{Department of Mathematics, University of California, Berkeley}
\email{david\_gonzalez@berkeley.edu}
\address{Department of Mathematics, University of California, Berkeley {\normalfont and} Institute of Discrete Mathematics and Geometry, Technische Universit\"at Wien}
\email{dino@math.berkeley.edu}

\thanks{The work of Rossegger was supported by the European Union's Horizon 2020 Research and Innovation Programme under the Marie Sk\l{}odowska-Curie grant agreement No. 101026834 — ACOSE}

\begin{document}
\maketitle
\begin{abstract}
  We study possible Scott sentence complexities of linear orderings using two
  approaches. First, we investigate the effect of the Friedman-Stanley embedding
  on Scott sentence complexity and show that it only preserves $\Pinf{\alpha}$
  complexities.
  We then take a more direct approach and exhibit linear orderings of all Scott
  sentence complexities except $\Sinf{3}$ and $\Sinf{\lambda+1}$ for $\lambda$ a
  limit ordinal. We show that the former can not be the Scott sentence
  complexity of a linear ordering. In the process we develop new techniques
  which appear to be helpful to calculate the Scott sentence complexities of
  structures.
\end{abstract}

In this article we give a comprehensive Scott analysis of the class of linear
orderings using three definitions of rank. Given a structure $\A$ the
\emph{(parameterless) Scott rank} $\SR(\A)$, is the least $\alpha$, such that
the automorphism orbits of all tuples in $\A$ are $\Sinf{\alpha}$ definable; the
\emph{parameterized Scott rank} $\pSR(\A)$, is the least ordinal $\alpha$ such
that there is a parameter $\bp\in A^{<\omega}$ such that $SR((\A,\bp))=\alpha$;
and the \emph{Scott sentence complexity} of $\A$, $SSC(\A)$ is the least of the
complexities $\Pinf{\alpha}$, $\Sinf{\alpha}$, $\dSinf{\alpha}$ such that $\A$
has a Scott sentence of one of these complexities. While the latter is not an
ordinal, it was shown in~\cite{alvir2020a}, that the Scott sentence complexity
is indeed a well-defined notion of rank. \cref{table:invariants}, taken
from~\cite{montalban2021}, shows the relationship between these invariants. As
can be seen from the definitions of the two Scott ranks and
\cref{table:invariants}, calculating the Scott sentence complexity of a
structure comes down to calculating the complexity of the automorphism orbits of
the structure's tuples in all but the limit ordinal cases. The case of limit ordinals is robustly
treated for the first time in this article and we discuss it more in \cref{sec:fs-ssc}. We refer the reader
to~\cite{montalban2021} for proofs and background on Scott analysis and many of
the tools used in this article.

\begin{table}[ht]
\centering
    \begin{tblr}{c|c|c|c}
    \hline
    SSC                &    pSR   &SR          & complexity of parameters\\
    \hline\hline
    $\Sinf{\alpha+2}$  & $\alpha$ & $\alpha+2$ & $\Pinf{\alpha+1}$\\
    $\dSinf{\alpha+1}$ & $\alpha$ & $\alpha+1$ & $\Pinf{\alpha}$\\
    $\Pinf{\alpha+1}$  & $\alpha$ & $\alpha$   & none\\
    \hline
    \SetCell[r=1,c=4]{c} $\alpha$ limit &&&\\
    \hline
    $\Sinf{\alpha+1}$  & $\alpha$ & $\alpha+1$ & $\Pinf{\alpha}$\\
    $\Pinf{\alpha}$    & $\alpha$ & $\alpha$   & none\\
    \hline
  \end{tblr}

  \caption{\cite[Table 1]{montalban2021}. Relationship of the different Scott invariants.
  The last column contains the complexity of the automorphism orbit of the
parameters involved in the parameterized Scott rank.}
\label{table:invariants}
\end{table}

By calculating the Scott sentence complexities occurring in a class of structures
we obtain a detailed picture of the descriptive complexity of its isomorphism relation.
By the Lopez-Escobar theorem, the Scott sentence complexities correspond to the
Borel complexity of the orbits, and it is well known that the isomorphism relation on a class of
structures is Borel if and only if the Scott ranks in the class are bounded
below $\omega_1$.

The class of linear orderings is a particularly interesting object of study as
it is Borel complete---for every class of structures $\mf K$ there is a
Borel function $f:Mod(\mf K)\to Mod(\LO)$ such that for all $\A,\B\in \mf K$,
$\A\cong \B$ if and only $f(\A)\cong f(\B)$~\cite{friedman1989}.
But the class of linear orderings is not complete for stronger forms of reduction such as faithful Borel
reducibility or effective bi-interpretability~\cite{gao2001}. 
If one dives into the realm of computable structure theory, one quickly sees
that linear orderings are far from being "structurally complete". For example,
Richter's result that no linear ordering can code a non-computable subset
of $\omega$ in its isomorphism type, shows that linear orderings are quite weak
in terms of coding power \cite{Ric81}. However, if one turns to non-effective notions such as Scott
rank and Scott sentence complexity, the situation becomes less clear. It is not
at all difficult to produce linear orderings of Scott rank $\alpha$ for $\alpha$
any countable ordinal, in fact Ash shows that one can find such examples by
considering well-orderings~\cite{ash1986}. Even among computable linear
orderings one can find every possible Scott rank; Harrison~\cite{harrison1968} and recently
Calvert, Goncharov, and Knight~\cite{calvert2007} exhibited linear orderings of
Scott rank $\ock+1$, and $\ock$ repectively. This beckons the question whether the
same phenomenon appears if one considers the more fine-grained notion of Scott
sentence complexity.

There are two approaches one might consider to study the possible Scott sentence
complexities in a specific class of structures. One is to study existing
Borel reductions, such as the Friedman-Stanley embedding~\cite{friedman1989} and
whether this embedding preserves Scott sentence complexities. We show that this
approach fails in \cref{sec:fs-ssc}. Our results
show that the Friedman-Stanley embedding does not preserve Scott sentence
complexity, in fact, for $\alpha>\omega$, if $SSC(\A)=\Sinf\alpha$, then,
$SSC(FS(\A))=\Pinf{\alpha+1}$ (This is true for finite $\alpha$ modulo a
constant). The picture is different if $SSC(\A)=\Pinf{\alpha}$, then
$SSC(FS(\A))=\Pinf{\alpha}$ even in the case when $\alpha$ is a limit ordinal
(again modulo a constant if $\alpha$ is finite). These results involve a deep
analysis on the effect of parameters on the tree-of-tuples of a structure and
a novel characterization of structures with Scott sentence complexity
$\Pinf{\lambda}$ using what we call unstable $\lambda$-sequences. 
In \cref{sec:universality} we show that our choice to study the Friedman-Stanley
embedding was not arbitrary; there are no better Borel embeddings into the class
of linear orderings in terms of preservation of Scott sentence complexity. In order to show 
this, we develop a concept we call $\alpha$-universality. This tool provides
upper bounds for the number of structures with particular Scott sentence complexities.

In \cref{sec:notssc} and \cref{sec:ssc} we take another more direct approach and
try to find examples of linear orderings of given Scott sentence complexities.
A summary of our findings can be found in the table below.

  \begin{table}[ht]
    \begin{tblr}{c|c|c}
      SSC                &   LO  & Reference\\
      \hline\hline
      $\Pinf{1}$ & + & $1$-element l.o.\\
      $\dSinf{1}$& + & $n$-element l.o. ($n>1$)\\
      $\Pinf{2}$ & + & $\Q$\\
      $\Sinf{3}$ & - & \cref{thm:nosigma3}\\
      \hline
      \SetCell[r=1,c=3]{c} successor $\alpha$: \\
      \hline
      $\Sinf{\alpha+3}$  & + & \cref{cor:SSCsuccessors}\\
      $\dSinf{\alpha}$  & + & \cref{cor:SSCsuccessors}\\
      $\Pinf{\alpha}$  & + & \cref{cor:SSCsuccessors}\\
      \hline
      \SetCell[r=1,c=3]{c} $\lambda$ a limit ordinal:\\
      \hline
      $\Pinf{\lambda}$ & + & \cref{prop:scatteredlimit}\\
      $\Sinf{\lambda+1}$&? & \cref{ques:lambda+1}\\
      $\Sinf{\lambda+2}$&+ & \cref{thm:Slim+2}\\
      $\Sinf{\lambda+3}$&+ & \cref{thm:Slim+3}\\
    \end{tblr}
    \caption{Scott sentence complexities of linear orderings. Complexities not
    in the table are impossible for structures in general.}
    \label{table:results}
  \end{table}
\subsection{Remarks on effectivity}
While this article's primary concern is not the effectiveness of the presented
results, many of our results in this article can be effectivized without much effort. The
interested reader can find further discussions on this matter in subsections at the end 
of \cref{sec:fs-ssc,sec:ssc}. Included are comments on potential difficulties and open 
questions that arise from these efforts.

\subsection{Preliminaries}
We assume that the reader is familiar with the basic techniques of Scott
analysis and with linear orderings in general. For background we refer to
Montalb\'an's upcoming book~\cite{montalban2021} and the standard reference by
Ash and Knight~\cite{ash2000}. Let us highlight the following lemma
which we will use as basic tools in our proofs without reference.
\begin{lemma}[{\cite[Lemma 15.8]{ash2000}}]\label{lem:bfandpartitions}  Suppose $\A,\B\in\LO$. Then $\A\leq_1 \B$ if and only if $\A$ is infinite or at least as large as $\B$. For $\alpha>1$, $\A\leq_\alpha \B$ if and only if for every $1\leq\beta<\alpha$ and every partition of $\B$ into intervals $\B_0,\dots,\B_n$, with endpoints in $\B$, there is a partition of $\A$ into intervals $\A_1,\dots,\A_n$ with endpoints in $\A$, such that $\B_i\leq_\beta \A_i$.
\end{lemma}
Notice that this lemma particularly implies that $(\A,\ba)\leq_\alpha (\B,\bb)$
if and only if $(-\infty, a_0)\leq_\alpha (-\infty, b_0)$,
$(a_{|\ba-1|},\infty)\leq_\alpha (b_{|\bb-1|},\infty)$ and for all $0\leq
i<|\ba-1|$ $(a_i,a_{i+1})\leq_\alpha (b_i,b_{i+1})$~\cite[Lemma 15.7]{ash2000}.

 \section{The Friedman-Stanley embedding and Scott sentence complexity}\label{sec:fs-ssc}
Friedman and Stanley~\cite{friedman1989} showed that the class of linear orderings is Borel complete
by defining a computable operator $\Phi$ that takes a structure in a fixed
vocabulary as input and outputs a linear ordering such that for all $\A,\B$,
$\A\cong\B$ if and only if $\Phi(\A)\cong \Phi(\B)$. This reduction is commonly
referred to as the Friedman-Stanley embedding. It proceeds in two steps.
\begin{enumerate}
   \tightlist
   \item\label{item:tree-of-tuples} Given a structure $\A$ it produces a labeled
      tree $\T_\A$, the
      tree of tuples of $\A$.
   \item From $\T_\A$ it produces a linear ordering $L_\A=\Phi(\A)$.
\end{enumerate}
The Friedman-Stanley embedding has been heavily studied both by descriptive set
theorists and computable structure theorists. Closely connected to our
investigation are Harrison-Trainor and Montalbán's study of the tree of
tuples constructions~\cite{harrison-trainor2022a} and Knight, Soskova and
Vatev's analysis of possible reductions from arbitrary classes of structure to
linear orderings~\cite{knight2020}. One of the results obtained in both these
articles is that one can not embed graphs into linear orderings uniformly using
$L_{\omega_1\omega}$ formulas. This raises the question whether the Friedman-Stanley embedding or any
reduction reducing graphs to linear orderings can preserve Scott sentence
complexity. We consider the first question in this section by analyzing the two
steps of the reduction.
\subsection{The tree of tuples and Scott sentence complexity}
\begin{definition}\label{def:replabeledtreed}
   A \emph{replicated labeled tree} consists of a tree $(T,\succeq)$ with a parent function and
   labeling function $l:T\to \omega$ that satisfies that for every $\sigma\in T$
   with parent $\tau$, there exist infinitely many children $\tilde\sigma$ of
   $\tau$ such that $\T_\sigma\cong \T_{\tilde\sigma}$.
\end{definition}
\begin{definition}\label{def:tree-of-tuples}
   Given a structure $\A$ let $T_\A$ be the labeled tree consisting of all
   tuples from $\A$ ordered by inclusion where each tuple $\ba$ is labeled by a
   natural number coding its finite atomic diagram $D_\A(\ba)$ and the length of
   $\ba$. The \emph{tree of tuples $\T_\A$} is obtained by replicating every branch in
   $T_\A$ infinitely many times.
  \end{definition}

Looking at the construction of $\T_\A$ we can see that the
Friedman-Stanley embedding does not preserve much structure. Say for instance
that $\A$ is rigid (i.e., it has trivial automorphism group); the tree $\T_\A$
will have many automorphisms, and thus the automorphism groups of $\A$ and
$\T_\A$ will behave quite differently.

On the other hand, we start by showing that the Scott ranks of the tree of tuples of a structure
and the Scott rank of the structure are equal. Towards this we need a few
lemmas.
\begin{lemma}\label{lemma:translation} 
   Let $\A$ be a structure in vocabulary $\tau$.
   \begin{enumerate}
      \tightlist
      \item\label{eq:pushforward} Given a $\Pinf{\alpha}$ $\tau$-formula $\phi$, there is a
         $\Pinf{\alpha}$ formula $T_\phi$ in the vocabulary of labeled trees
         such that for any $\ba\in \A$ and every $\sigma\in \T_\A$ with
         $\sigma\succeq \ba$ 
         \[\A\models\phi(\ba)\iff \T_\A\models T_\phi(\sigma).\]
      \item\label{eq:pullback} Given a $\Pinf{\alpha}$ formula $\psi$ in the
         vocabulary of labeled trees, there is a $\Pinf{\alpha}$ $\tau$-formula
         $\psi^*$ such that for any tuple $\bar{\sigma}=(\ba_1,\dots,\ba_n)\in
         \T_\A$ \[\T_\A\models \psi(\bar\sigma) \iff
         \A\models\psi^*(\ba,\dots,\ba_n).\]
      \end{enumerate}
\end{lemma} 
\begin{proof}
   We first show \cref{eq:pushforward} starting with quantifier free formulas.
   Let $\phi$ be quantifier free with free variables $x_1,\dots x_n$ in
   disjunctive normal form. For each disjunct $\phi_i$ let $n_i$ be the index of
   the largest occuring variable and let $T_{\phi_i}$ be the formula expressing
   that the height of $x$ in the tree is at least $n_i$ and that
   $\phi_i\subseteq l(x)$. Then $T_\phi=\bigvee T_{\phi_i}$. The construction of
   $\T_\A$ guarantees that for $\sigma\succeq\ba\in \A$, $\T_\A\models
   T_\phi(\sigma)$ if and only if $\A\models \phi(\ba)$. Now, assume that we
   have a translation for $\phi(\bx,y)$ into $T_\phi(x)$. If $\psi=\exists y
   \phi(\bx,y)$, then let \[T_\psi= (\exists y\succeq x)\, T_\phi(y).\] We have
   that $\A\models \exists y\,\psi(\ba,y)$, if and only if there is $b\in \A$
   such that $\A\models\psi(\ba,b)$ if and only if there is $\sigma\succeq \ba
   \in \T_\A$ $\T_\A\models T_\psi(\sigma)$ if and only if $\T_\A\models
   T_\psi$. This proves the successor case. The limit case is immediate.

   We now prove \cref{eq:pullback} by defining a family of Turing-computable
   embeddings from the class of structures containing $\A$ with parameters to
   labeled trees with parameters. Each of these embeddings will in some sense
   extend the Friedman-Stanley embedding that acts without parameters. Fix a
   tuple $\bar\sigma=(\ba_1,\cdots,\ba_n)\in \T_\A$. Consider any structure $\B$
   in the class of $\A$ with parameters $\bp$ such that
   $|\bp|=|\ba_1\dots\ba_n|$. We now construct the embedding
   $\Psi_{\bar\sigma}$. Non-uniformly fix a partition
   $\bp=\bp_1,\dots\bp_n$ with $|\bp_i|=|\ba_i|$ and construct the tree of
   tuples of $\B$ as normal. To decide what elements become parameters in
   $\T_\B$, look for the lexicographical least elements $\tau_i$ coding the
   $\bp_i$ such that the finite substructure of $\T_\B$ described by $\tau$ is
   isomorphic to the one described by $\sigma$ in $\T_\A$. It is not difficult to check that this is a Turing computable
   embedding. Now, for any formula $\psi$ in the language of labeled trees, any
   $\B$ in the class of $\A$ and any $\bar\sigma=(\bb_0\dots\bb_n)\in \T_\B$ by
   the pull-back theorem for $\Psi_{\bar\sigma}$ there is $\psi^*$ of the same
   complexity as $\psi$ such that $\B\models \psi^*(\bb_0\dots\bb_n)$ if and
   only if $\T_\B\models \psi(\bar\sigma)$. As the image of $\Psi_{\bar\sigma}$
   contains the parameterized tree $(\T_\A,\bar\sigma)$, this proves
   \cref{eq:pullback}.
\end{proof}
\cref{lemma:translation} might suggest that the tree of tuples reduction
preserves Scott sentence complexity. However, as we will see, it does not.
The reason for this is the following. Let $K$ be a class of structures and look at $T(K)$, the set of trees-of-tuples
of the elements of $K$, then within $T(K)$, Scott sentences are preserved. However,
$T(K)$ is not a Borel subset of the class of labeled trees and thus the
Scott sentence complexity of $\T_\A$ is not a priori connected to the Scott
sentence complexity of $\A$.
\begin{proposition}\label{prop:treeofstructuresbf}
Given
$\bar\sigma=(\ba_1,\dots,\ba_n),\bar\tau=(\bb_1,\dots,\bb_n)\in \T_\A$, if
$\bar\sigma\leq_\alpha \bar\tau$, then for
each $i<n$, $\ba_i\leq_\alpha \bb_i$.
\end{proposition}
\begin{proof}
We prove this by contraposition.
Assume that for some $i$, $\ba_i\not\leq_\alpha \bb_i$. 
This is the same thing as saying that there is a $\Pinf{\alpha}$ formula $\phi$
with $\A\models\phi(\ba_i)$ and $\A\models\lnot\phi(\bb_i)$.
Using the first part of \cref{lemma:translation} we get that $\T_\A\models
T_\phi(\ba_i)$ and $\T_\A\models\lnot T_\phi(\bb_i)$.
It immediately follows that $\bar\sigma\not\leq_\alpha \bar\tau$ as desired.
\end{proof}
\begin{theorem}\label{thm:treeofstructuressr}
For any structure $\A$, $\SR(\A)=\SR(\T_\A)$. 
\end{theorem}
\begin{proof}
Let $\A$ be a structure such that $SR(\T_\A)=\gamma$. Consider a tuple $\ba\in \A$. Let
$\psi\in\Sinf{\gamma}$ define the orbit of $\ba\in \T_\A$. By \cref{eq:pullback}
in \cref{lemma:translation} we have for all $\bb\in \A$ \[\A\models
\psi^*(\bb)\iff \T_\A\models \psi(\bb) \iff \bb\cong \ba.\] Therefore, $\psi^*$
defines the orbit of $\ba$ as desired and hence, $\SR(\A)\leq\SR(\T_\A)$. 

We now show $\SR(\T_\A)\leq\SR(\A)=\gamma$. Let $\bar\sigma=\ba_1\dots\ba_n$ be
a tuple in $\T_\A$. The automorphism orbit of each  $\ba_i$ is $\Sinf{\gamma}$
definable in $\A$, thus by \cref{eq:pushforward} in \cref{lemma:translation}
there is a $\Sinf{\gamma}$ formula, $\chi_i$,  true only of elements $\tau$ extending $\sigma$ 
that code elements automorphic to $\ba_i$ in $\A$. Now, let
$\psi$ describe the finite subtree containing precisely the elements of
$\bar\sigma$ as its leaves and let $\phi$ by the conjunction of $\psi$ together
with the $\chi_i$'s. Note that
because the labels encode the height of the labeled element in the tree, $\psi$
can be taken $\Sinf{1}$. Let $\bar \tau=\bb_1\dots\bb_n$ satisfy $\phi$. Then,
for every $\bb_i$, the subtrees $T_{\bb_i}$ rooted at $\bb_i$ are isomorphic to
the subtrees of $\ba_i$ and as the tree $\T_\A$ is replicated we can extend
these isomorphisms to an automorphism mapping $\bar\sigma\to \bar\tau$.
\end{proof}
\cref{thm:treeofstructuressr} can be viewed as a strengthening of a result of
Gao. He showed, using a different definition of Scott rank, that the ranks of
$\A$ and $\T_\A$ can be at most $\omega^2$ apart~\cite{gao2001}. 

We now turn our attention to the paramaterized Scott rank.
Unlike the unparamaterized case, paramaterized Scott rank will not be preserved.
To see this, we first need the following definitions. 
Given a tree $T$ with root $r$ and a tuple $\bar{c}=(c_1,\cdots,c_n)\in T$, we
say that $y\in T$ is \textit{somewhat comparable} to $\bar{c}$ if $\exists z,i ~
z\neq r\land c_i\succeq z \land y\succeq z$.
A tuple is somewhat comparable to $\bar{c}$ if all of its points are.
Similarly, we say that a point is \textit{entirely incomparable} to $\bar{c}$ if it is not somewhat comparable to $\bar{c}$ and that a tuple is entirely incomparable to $\bar{c}$ if all of its points are entirely incomparable to $\bar{c}$.

\begin{lemma}\label{lem:bftos_noparameters}
For $\T_\A$ a tree of tuples, every $\alpha\in\omega_1$ and tuples $\ba$ and $\bb$ that are entirely incomparable to $\bp$,
\[\ba\bp\leq_\alpha\bb\bp \iff \ba\leq_\alpha\bb.\]
\end{lemma}
\begin{proof}
We go by transfinite induction on $\alpha$.
The $\alpha=0$ case is immediate, as is the limit case.
Assume the result for $\alpha=\beta$.
If $\ba\bp\leq_{\beta+1}\bb\bp$, then $\ba\leq_{\beta+1}\bb$ follows immediately.
Now assume that $\ba\leq_{\beta+1}\bb$.
We wish to show that $\ba\bp\leq_{\beta+1}\bb\bp$.
Consider some extension $\bc$.
We can split $\bc$ into the points somewhat comparable to $\bp$, named $\bc_p$, and the points entirely incomparable to $\bp$, named $\bc_i$ (one of these may be the empty tuple).
Because $\ba\leq_{\beta+1}\bb$, there is a tuple $\bd_i$ such that $\ba\bd_i\geq_{\beta}\bb\bc_i$.
Without loss of generality (using an automorphism that fixes $\ba$ and moves $\bp$ to something entirely incomparable if needed), $\bd_i$ is entirely incomparable to $\bp$. 
Indeed, as somewhat comparability is an equivalence relation, we can note that $\ba\bd_i$ and $\bb\bc_i$ are entirely incomparable with $\bp\bc_p$.
In particular, the induction hypothesis gives that
\[\ba\bd_i\bp\bc_p\geq_{\beta}\bb\bc_i\bp\bc_p.\]
Rearranging the tuples gives that
\[\ba\bp\bd_i\bc_p\geq_{\beta}\bb\bp\bc\bc_p\]
which shows that $\ba\bp\leq_{\beta+1}\bb\bp$, as desired.
\end{proof}
\cref{lem:bftos_noparameters} can also be proved using a general result of
Steel~\cite{steel1978}.
\begin{lemma}
For every tree of tuples $\T_\A$, $\SR(\T_\A)=\SR_p(\T_\A)$.
\end{lemma}

\begin{proof}
Consider a tuple $\ba$ that is $\alpha$-free in $\T_\A$.
Without loss of generality (by picking parameters automorphic to $\bp$) we have that $\ba$ is entirely incomparable with $\bp$.
We demonstrate that $\ba$ is also $\alpha$-free in $(\T_\A,\bp)$.

Consider $\bb\in \T_\A$ and $\beta<\alpha$.
We wish to show that there are $\ba',\bb'$ such that
$\ba\bb\bp\leq_\beta\ba'\bb'\bp$
and
$\ba\bp\not\leq_\alpha\ba'\bp.$

We can split $\bb$ into the points somewhat comparable to $\bp$, named $\bb_p$, and the points entirely incomparable to $\bp$, named $\bb_i$ (one of these may be the empty tuple).
Using the $\alpha$-freeness of $\ba$ in $\T_\A$, we may find $\ba'$ and $\bb_i'$ (entirely incomparable to $p$ via an automorphism fixing $\ba$) such that
$\ba\bb_i\leq_\beta\ba'\bb_i'$
and $\ba\not\leq_\alpha\ba'.$

It follows from the previous lemma that
$\ba\bb_i\bp\bb_p\leq_\beta\ba'\bb_i'\bp\bb_p$
and
$\ba\bp\not\leq_\alpha\ba'\bp.$
Rearranging the tuples we get that
$\ba\bb\bp\leq_\beta\ba'\bb'\bp$
and
$\ba\bp\not\leq_\alpha\ba'\bp.$
Therefore, $\ba$ is also $\alpha$-free in $(\T_\A,\bp)$.

In particular, if there are no $\alpha$-free tuples in $(\T_\A,\bp)$ then there
are no $\alpha$-free tuples in $\T_\A$. 
Hence, $\SR(\T_\A)=\SR_p(\T_\A)$. \end{proof}
The previous two results can be used to understand the effect of the
Friedman-Stanley embedding on Scott sentence complexity in all but the limit
cases. We use the following new technique to deal with these cases.

\begin{definition}
For a structure $\A$ and a limit ordinal $\lambda$, a
\textit{$\lambda$-sequence} in $\A$ is a set of tuples $\by_i\in \A$ for
$i\in\omega$ such that $\by_i\equiv_{\alpha_i}\by_{i+1}$ for some fundamental
sequence $(\alpha_i)_{i\in\omega}$ for $\lambda$.  We say that a $\lambda$-sequence is \textit{unstable} if $y_i\not\equiv_{\alpha_{i+1}}y_{i+1}$.
\end{definition}

\begin{lemma}
Consider a structure $\A$ with $\SR(\A)=\lambda$ for $\lambda$ a limit ordinal. The Scott sentence complexity
of $\A$ is $\Pinf{\lambda}$ if and only if there are no unstable
$\lambda$-sequences in $\A$.
\end{lemma}

\begin{proof}
We start with the right to left implication and argue by contraposition. Assume
$\SR(\A)=\lambda$, yet $\A$ does not have a $\Pinf{\lambda}$ Scott sentence.
Then there is a structure $\B$ such that $\A\equiv_\lambda \B$ yet $\A\not\cong
\B$. Because $\A$ has Scott rank $\lambda$ this is equivalent to saying that
$\A\equiv_\lambda \B$ yet $\A\not\leq_{\lambda+1} \B$. In other words, there is
a tuple $\bx\in \B$ such that for any $\bz\in \A$, $\bz\not\equiv_\lambda \bx$.
However, as $\A\equiv_\lambda \B$, given a fundamental sequence
$(\alpha_i)_{i\in\omega}$ for $\lambda$, for every $i$ there is some $\by_i\in \A$ such that
$\by_i\equiv_{\alpha_i} \bx$. It follows that the $\by_i$ are a
$\lambda$-sequence. Furthermore, given any $i$, there must be some $k>i$ such
that $\by_i\not\equiv_{\alpha_{k}}\by_{k}$. Otherwise, $\by_i\equiv_\lambda \bx$,
a contradiction to the choice of $\bx$. The above $k$ can be taken to equal
$i+1$ if we thin out the sequence by removing intermediate elements between $i$
and $k$. In particular, the sequence $(\by_i)_{i\in\omega}$ can be taken to be unstable.

We now prove the left to right implication, again by contraposition. Consider a
structure $\A$ with an unstable $\lambda$-sequence $(\by_i)_{i\in\omega}$. Consider the sequence of
structures given by $(\A,\by_i)_{i\in\omega}$. By~\cite[Lemma
XII.6]{montalban2021a}, there is a structure $\B$ and tuple $\bz\in \B$ such
that for each $i$, $(\B,\bz)\equiv_{\alpha_i}(\A,\by_i)$. Note that, in
particular, this means that $\B\equiv_\lambda \A$. Furthermore, for each $i$
there is some $k>i$ such that $\alpha_k\geq 2\alpha_i+3$, as the difference
between $2\alpha_i+3$ and $\alpha_i$ is finite. For this $k$, we can observe that 
\[(\A,\by_k)\models \exists \bx \ \bx\not\equiv_{\alpha_{i+1}} \by_k \land
\bx\equiv_{\alpha_i} \by_k.\]
By~\cite[Lemma VI.14]{montalban2021a} this is a
$\Sinf{2\alpha_i+3}$ sentence that is computable in $\A$. As
$(\B,\bz)\equiv_{\alpha_k}(\A,\by_k)$, by choice of $k$,
\[(\B,\bz)\models \exists \bx ~ \bx\not\equiv_{\alpha_{i+1}} \bz \land
\bx\equiv_{\alpha_i} \bz.\]
Now consider the automorphism orbit of $\bz$ inside $\B$. It cannot be defined
by a $\beta$ formula for $\beta<\lambda$, as if we take $\alpha_i>\beta$, by the
above observation, there is a tuple that is not in the automorphism orbit of
$\bz$ that is $\alpha_i$-equivalent to $\bz$. This means that $SR(\B)>\lambda$ so
$\B\not\cong \A$. As $\B\equiv_\lambda \A$, there is no $\Pinf{\lambda}$ Scott sentence for $A$.
\end{proof}

The former lemma be better understood as an internal compactness property that
is present in structures with Scott sentence complexity  $\Pinf\lambda$. For the sake of completeness we include this equivalent formulation here.

\begin{corollary}
Consider a structure $\A$ with $\SR(\A)=\lambda$ for $\lambda$ a limit ordinal. The Scott sentence complexity
of $\A$ is $\Pinf{\lambda}$ if and only if for any fundamental sequence
$(\delta_n)_{n\in\omega}$ for $\lambda$ and corresponding complete $\delta_n$-types $p_n$ with $p_n\subset p_{n+1}$ the following holds
\[\A\models \exists \bx \bigwwedge_{n}\delta_n\text{-tp}(\bx)=p_n \iff \A\models
\bigwwedge_{n}\exists \bx\ \delta_n\text{-tp}(\bx)=p_n.\]
\end{corollary}

\begin{proof}
Say $\A$ has Scott sentence complexity $\Pinf{\lambda}$. If $\A\models 
\bigwwedge_{n}\exists \bx\ \delta_n\text{-tp}(\bx)=p_n$ then, by construction, the witnesses to
$\delta_n\text{-tp}(\bx)=p_n$, call them $\by_n$, form a $\lambda$-sequence. If
$\by_n\not\equiv_{\delta_{k}} \by_{k}$ for $k>n$ on a cofinal subset, by
thinning out the sequence we obtain an unstable sequence, a contradiction.
Therefore, there is some $K$ for which $\by_n\equiv_{\delta_{k}} \by_{k}$ for
all $k>n\geq K$. Then $\by_K$ serves as the needed witness for
$\A\models \exists \bx \bigwwedge_{n}\delta_n\text{-tp}(\bx)=p_n$. 

On the other hand, if $\A$ does not have Scott sentence complexity
$\Pinf{\lambda}$ it must have some unstable sequence, $(z_n)_{n\in\omega}$. If
we take the $\delta_n$-type of $\bz_n$ to be $p_n$, then we know that $\A\models
\bigwwedge_{n}\exists \bx~ \delta_n\text{-tp}(\bx)=p_n$. However, if there is
some witness to $\A\models \exists \bx\ \bigwwedge_{n}\delta_n\text{-tp}(\bx)=p_n$
we have that $\bx\equiv_{\delta_n}\bz_n$ yet
$\bx\not\equiv_{\delta_{n+1}}\bz_{n+1}$. Therefore the orbit of $\bx$ is not
$\Sinf{\lambda}$ definable, a contradiction as $\SR(\A)=\lambda$.
\end{proof}

Note that this cannot be upgraded from complete types to general formulas. In other
words, there may be a sequence of properly $\Pinf{\delta_n}$ formulas
$\phi_n$ such that $\phi_{n+1}\to\phi_n$ and $\A\models \bigwwedge_{n}\exists
\bx~ \phi_n(\bx)$ yet $\A\models \lnot \exists \bx\bigwwedge_{n} \phi_n(\bx)$.
To see this, we need not look past our example of a $\Pinf{\lambda}$ linear
order for $\lambda$ a limit given in \cref{prop:scatteredlimit}. In particular the following formulas are
a counterexample satisfied by this ordering.
\[\phi_n(z):=\exists\bx,\by \bigwedge_{i<n} (x_i,y_i)\cong \Z^{\delta_i}\land y_i<z\]

We now note that unstable sequences transfer between a structure and its tree of tuples.
This will allow us to resolve the final Scott sentence complexity ambiguity.
In order to do this, we first need the following technical lemma.

\begin{lemma}\label{lemma:factoring}
Let $S=(x_1,\cdots,x_n)$ and $T=(y_1,\cdots,y_n)$ be (finite) isomorphic
downward closed subsets of $\T_\A$.
$S\leq_\alpha T$ if and only if for all $i$, $x_i\leq_\alpha y_i$. 
\end{lemma}

\begin{proof}
If $S\leq_\alpha T$, $x_i\leq_\alpha y_i$ for any $i$ follows immediately from the monotonicity of $\leq_\alpha$.

Say that for all $i$, $x_i\leq_\alpha y_i$.
For $z\in \T_\A$ Let $U(z)$ be all elements above $z$ and note that
for all $i$, $U(x_i)\leq_\alpha U(y_i)$.
This holds because any $\Pinf{\alpha}$ formula about $U(x_i)$ can be transformed into a $\Pinf{\alpha}$ formula about $x_i$ by restricting all quantifiers to being above $x_i$.

Let $U_S(x_i) = \{y\in T_A\vert y\geq x_i \land \lnot\exists j ~ x_j>x_i \land y\geq x_j\}$ and define $U_T(y_i)$ analogously.
Because $S$ is downward closed, the $U_S(x_i)$ partition $\T_\A$ into disjoint sets (and similarly for $T$).
Because the branches of $T_A$ are infinitely replicated, we can observe that $U(x_i)\cong U_S(x_i)$ and $U(y_i)\cong U_T(y_i)$.
In particular, this means that for all $i$, $U_S(x_i)\leq_\alpha U_T(y_i)$.

For the sake of contradiction, say that $S\not\leq_\alpha T$.
(Note that $\alpha\neq0$, as $S$ and $T$ are isomorphic.)
This means there is a $\beta<\alpha$ and $\bd\in T_A$ such that for all $\bc\in T_A$
$(S,\bc)\not\geq_\alpha (T,\bd)$.
Write $\bd=(\bd_1\cdots\bd_n)$ when $\bd_i$ is the (possibly empty) subset of
$\bd$ that is in $U_T(y_i)$. For every $i$, let $\bc_i$ be the $\exists$-player's winning response to $\bd_i$ in $U_S(x_i)$ and let $\bc=(\bc_1\cdots \bc_n)$.
This is a valid move because the elements in  $U_T(y_i)$ and $U_T(y_j)$ for distinct $i$ and $j$ are never comparable to each other. 
As the game proceeds, the $\exists$-player may continue to play according to this strategy for the same reason.
When the game reduces to the $0$ level, it is assured that the $\exists$-player has won on all of the $U_S(x_i)$ and  $U_T(y_i)$.
In other words, the relations of the final tuple within these structures is the same.
As observed above, there are never any relations between points in distinct $U_S(x_i)$ or  $U_T(y_i)$.
Therefore, the relations between the structures are also the same.
Thus we have described a winning strategy for the $\exists$-player contradicting
that $S\not\leq_\alpha T$
\end{proof}

\begin{lemma}
A structure $\A$ has an unstable $\lambda$-sequence if and only if its tree of
tuples, $\T_\A$ does.
\end{lemma}

\begin{proof}
   Consider an unstable $\lambda$-sequence in $\A$ given by $(\by_k)_{k\in\omega}$. 
Say each tuple is coded by the element $m_k\in \T_\A$, then \cref{lemma:translation} gives us at once that $m_k$ is an unstable $\lambda$-sequence. 

On the other hand, consider an unstable $\lambda$-sequence in $\T_\A$ given by the fundamental
sequence $(\delta_k)_{k\in\omega}$ and associated tuples $(\bz_k)_{i\in\omega}$. 
Let $S_k$ denote the finite tree of elements less than or equal to some point in
the tuple, $\bz_k$.
Note that the isomorphism type of $S_k$ is fixed so long as we have without loss of generality that $\delta_k>1$ for all $k$.
By the definition of unstable $\lambda$-sequences we have that
\[\bz_{k}\equiv_{\delta_{k}}\bz_{k+1} \text{  and  }
\bz_{k}\not\equiv_{\delta_{k+1}}\bz_{k+1}.\]
By picking all of the elements in $S_k$ or $S_{k+1}$ as the first move of a back-and-forth game, we obtain that
\[S_{k}\equiv_{\delta_{k-1}}S_{k+1} \text{  and  }
S_{k}\not\equiv_{\delta_{k}}S_{k+1}.\]
It follows from \cref{lemma:factoring} that by reducing to some reindexed
fundamental subsequence of $(\delta_k)_{k\in\omega}$, say $(\gamma_k)_{k\in\omega}$, we have that for some sequence of points $x_k \in S_k$
\[x_{k}\equiv_{\gamma_{k}}x_{k+1} \text{  and  }
x_{k}\not\equiv_{\gamma_{k+1}}x_{k+1}.\]
By \cref{lemma:translation}, the $\by_k$ coded by the $x_k$, satisfy 
\[\by_{k}\equiv_{\gamma_{k}}\by_{k+1} \text{  and  }
\by_{k}\not\equiv_{\gamma_{k+1}}\by_{k+1}.\]
Therefore, there is an unstable $\lambda$-sequence in $\A$.
Together, we obtain the desired result.
\end{proof}

\begin{corollary}\label{cor:fs_limitssc}
$A$ has Scott sentence complexity $\Pi_\lambda^{in}$ if and only if $T_A$ has Scott sentence complexity $\Pi_\lambda^{in}$.
\end{corollary}

\begin{corollary}
   If $\A$ has Scott sentence complexity $\Gamma_\beta^{\mathrm{in}}$,
   $\T_\A$ has Scott sentence complexity $\Pinf{\alpha}$ where $\alpha=\beta$ if $\Gamma=\Pi$ and $\alpha=\beta+1$ otherwise.
\end{corollary}

\begin{proof}
If $\SR(\A)=\SR(\T_\A)$ is not a limit ordinal, then it follows from
$\SR(\T_\A)=\SR_p(\T_\A)$ that $\T_\A$ has $\Pinf{\SR(\A)+1}$ Scott sentence
complexity.

If $\SR(\A)=\SR(\T_\A)=\lambda$ is a limit ordinal, then by \cref{cor:fs_limitssc}
$SSC(\A)=\Pinf{\lambda}$ if and only if $SSC(\T_\A)=\Pinf{\lambda}$.
\end{proof}

\subsection{From tree of structures to linear orderings}
We now turn our attention to the second step of the Friedman-Stanley embedding.
Following Friedman and Stanley~\cite{friedman1989} we define the linear ordering
$L(T)$ given a labeled tree as follows.
\begin{definition}\label{def:shufflesum}
   Let $T$ be a labeled tree with labeling function $l_T:T\to\omega$ and
   take the linear ordering $\eta^{<\omega}$, given by the lexicographic order
   on finite strings of elements in $\mathbb Q$. We first define a map
   $f:\eta^{<\omega}\to T$ by recursion as follows: Map the empty string to the
   root of $T$, then assuming we have defined $f(\sigma)$ we define $f$ on
   elements of the form $\sigma\concat q$ for $q\in \eta$ to be a map from
   $\eta$ to $\{\tau: \exists i\ \tau=\sigma\concat i\}$ such that
   $f^{-1}{\tau}$ is dense in $\{\sigma\concat q: q\in\eta\}$. Then, for every
   $\sigma\in T$, if $l(\sigma)=n$, then for every $x\in f^{-1}(\sigma)$,
   replace $x$ by the finite linear ordering of size $n+2$.
   If $T$ is the tree of structures $\T_\A$ of $\A$, then we refer to $L(T)$ as
   $L_\A$.
\end{definition}

\begin{lemma}\label{lem:bf_treetolo}
   For every $\alpha$ and ordered $\ba,\bb\in L_\A$ such that if $a_i$, or $b_i$ are from a labeled
   block, then every element in this block is in $\ba$, respectively $\bb$, $\ba\leq_{3+\alpha}\bb$ if
   and only if $f(\ba)\leq_{1+\alpha} f(\bb)$.
\end{lemma}
\begin{proof}
   The proof is by transfinite induction with the only interesting case being
   the base case. For this note that the label and predecessor
   relation of $\T_\A$ in $L_{\A}$ are both $\Sinf{3}$ and $\Pinf{3}$
   definable in $L_{\T_\A}$. Let $\phi$ be a $\Sinf{1}$ formula true of
   $\ba$, then this formula can be translated into a $\Sinf{3}$ formula true
   of $f(\ba)$ and thus $\ba\leq_3 \bb$ implies $f(\ba)\leq_1 f(\bb)$. 

   On the other hand, assume that $f(\ba)\leq_1 f(\bb)$ and without loss of
   generality that $\ba$ and $\bb$ contain elements from exactly two
   different blocks. Let us consider the back-and-forth game where the
   $\forall$-player plays elements $\bb^1$ on their first turn. For every
   $b^1_i$ the $\exists$-player looks for elements $a^1_i$ such that
   $l(f(a^1_i))=l(f(b^1_i))$, $a^1_i$ is in the same position in its block as
   $b^1_i$ and $\ba\ba^1$ and $\bb\bb^1$ are
   isomorphic in the tree-ordering. Note that they will find a suitable $\ba^1$
   because $f(\ba)\leq_1 f(\bb)$. Now on their last turn the $\forall$-player
   plays a tuple $\ba^2$ and the $\exists$-player has to respond playing
   $\bb^2$ such that $\ba\ba^1\ba^2\leq_1 \bb\bb^1\bb^2$. All they need to do
   is ensure that the intervals in $\ba\ba^1\ba^2$ are at least as large as the
   intervals in $\bb\bb^1\bb^2$. Between any two elements in different blocks in
   $L_\A$, there are blocks coding paths in $\T_\A$. Thus, between any two
   elements in different blocks there are blocks of arbitrary large size and
   thus the $\exists$-player can find elements satisfying this requirement.
\end{proof}
\begin{lemma}
   For any structure $\A$, $2+\SR(\T_\A)=\SR(L_\A)$. In
   particular, $\T_\A$ has a $\Pinf{1+\alpha}$ Scott sentence if and only if
   $L_\A$ has a $\Pinf{3+\alpha}$ Scott sentence.
\end{lemma}
\begin{proof}
   For successor ordinals it is sufficient to note that \cref{lem:bf_treetolo} implies that a tuple $\ba$ in $L_\A$ is $(3+\alpha)$-free if and only if $f(\ba)$ is
   $(1+\alpha)$-free in $\T_\A$. Likewise, a tuple $\ba$ is
   $(1+\alpha)$-free in $\T_\A$ if and only if $f^{-1}(\ba)$ is
   $(3+\alpha)$-free in $L_\A$. 

   For $\alpha$ a limit, note that by \cref{lem:bf_treetolo} an
   $\alpha$-sequence for $L_\A$ is unstable if and only if the pull-back along
   $f$ is unstable in $\T_\A$.
\end{proof}
Combining everything we have developed in this section we obtain.
\begin{theorem}\label{thm:fsandscottrank}
   For $\alpha>\omega$, the Friedman-Stanley embedding preserves parameterized
   Scott rank, but it does not preserve parameterless Scott rank. In particular,
   for any given structure $\A$ and any ordinal $\alpha$:
   \begin{enumerate}
      \tightlist
      \item If $SSC(\A)=\Pinf{1+\alpha}$ then $SSC(L_\A)=\Pinf{3+\alpha}$
      \item If $SSC(\A)=\Sinf{1+\alpha}$ or $SSC(\A)=\dSinf{1+\alpha}$ then 
      $SSC(L_\A)=\Pinf{3+\alpha+1}$
   \end{enumerate}
\end{theorem}
Another useful characterization of structures with Scott sentence complexity
$\Pinf{\lambda}$ is that they are precisely those structures whose
$\Pinf{\lambda}$ infinitary theory is $\aleph_0$-categorical. Our results thus
show the following.
\begin{corollary}
   For any limit ordinal $\lambda$ and any structure $\A$, the $\Pinf{\lambda}$
   infinitary theory of $\A$ is $\aleph_0$-categorical if and only if the
   $\Pinf{\lambda}$ theory of $L_\A$ is $\aleph_0$-categorical.
\end{corollary}
The above corollary is related to a result by Calvert, Goncharov, and Knight
showed that there is a linear ordering of Scott rank $\ock$ by proving a
``rank-preservation property'' for computable embeddings~\cite{calvert2007}. We
will not define this property. What their arguments show is that the
Friedman-Stanley embedding preserves $\aleph_0$-categricity of the computable
infinitary theory of structures. As a structures computable infinitary theory is
$\aleph_0$-categorical if and only if it has Scott rank $\ock$ this shows that
the Friedman-Stanley embedding preserves Scott rank at that level
(see~\cite{alvir2021}).
\subsection{Remarks on effectivity}
Some Borel reductions maintain Scott rank externally in the sense that they take sufficiently complex Borel isomorphism classes and push them forward to classes of equal complexity in the image space.
In this case, Scott sentence complexity is preserved for sufficiently high
complexities.

\begin{proposition}\label{prop:push-forwardssc}
   If $\Phi:Mod(\sigma)\to Mod(\tau)$ is a Borel reduction such that
   $[\Phi(X)]_{\cong}$ is Borel and that eventually satisfies push-forward, i.e.,
   there is $\delta$ such that for $\gamma\geq\delta$ and
   $\psi\in\Pinf{\gamma}$ there is $\psi^*\in\Pinf{\gamma}$ such that
   $$\A\models \psi \iff \Phi(A)\models\psi^*$$
   then there is an ordinal $\alpha$ such that $\text{SR}(\A)\geq\alpha$ implies
   that $SSC(\A)=SSC(\Phi(\A))$.
\end{proposition}

\begin{proof}
Say $\Phi$ is $\bSigma_\lambda^0$-measurable and the image of $\Phi$ is
$\bPi_\beta^0$. Let $\chi\in\Pinf{\beta}$ describe the image of $\Phi$ in
$Mod(\tau)$. 
Let $\A$ have a $\Gamma^{\mathrm{in}}_\epsilon$ Scott sentence for
$\epsilon>\delta$, say $\phi$. Then $\phi^*\land \chi$ is a
$\Gamma^\mathrm{in}_\epsilon$ Scott sentence for $\Phi(\A)$. On the other hand
say $\Phi(\A)$ has a $\Gamma^{\mathrm{in}}_\epsilon$ Scott sentence for any
$\epsilon$, then by the pullback theorem $\A$ has a
$\Gamma^{\mathrm{in}}_{\lambda+\epsilon}$ Scott sentence. So for any $\A$ with
$\SR(\A)>\max(\lambda\cdot \epsilon, \delta,\beta)$, Scott sentence complexity
is preserved.
\end{proof}

Note that this proof has a straightforward effectivization. In particular, if
$\Phi$ is hyperarithmetic and $\psi$ computable implies that $\psi^*$ is
computable, then the complexity of the simplest computable formula that
describes the isomorphism type of $\A$ is also preserved.

However, as we have seen in \cref{thm:fsandscottrank}, the Friedman-Stanley
embedding does not have the property that Scott sentence complexity is eventually
preserved. We thus get the following corollary.

\begin{corollary}
   The image of the Friedman-Stanley embedding is not a Borel subset of the
   class of linear orderings.
\end{corollary}

It should be noted that we can observe that the class of trees of structures is
not Borel among the class of labeled trees directly. In particular, consider
$\T_\A$ and $\T_\B$ for structures with $\A\equiv_{\alpha+1} \B$ yet
$\A\not\cong \B$. We can now define $T_{\A,\B}:= \T_\A\sqcup_r\T_\B$, the tree
formed by identifying the roots of $\T_\A$ and $\T_\B$ and adding no other
relations between them. Given our established criteria for $\equiv_\beta$, it is
not difficult to show that $T_{\A,\B}\equiv_{\alpha}\T_\A\equiv_{\alpha}\T_\B$.
However, $T_{\A,\B}$ is not a tree of structures, as it lacks the property that
comeager many paths represent the same structure. By taking $\alpha$ arbitrarily
large we can observe that there can be no $L_{\omega_1\omega}$ formula defining
the class of tree of structures within the class of labeled trees.

This justifies why the proof of preservation of Scott rank proceeded as it did above.
In particular, Scott rank was treated "internally" by looking at the definitions of automorphism orbits inside of the structure.
Montalb\'an's theorem on the robustness of the Scott rank~\cite{montalban2015}
and the work of Alvir, Greenberg, Harrison-Trainor and
Turetsky~\cite{alvir2021} allowed us to take this approach, as there is an established correspondence between the internal characterization of Scott sentence complexity and the external one.
However, no such correspondence exists for computable formulas.
Therefore, the proof technique used does not make it clear if there is a preservation of computable Scott rank.

The closest result to an effective version of Montalb\'an's theorem is that of Alvir, Knight and McCoy~\cite{alvir2020a}. They prove the following.

\begin{theorem}\label{thm:alvirknightmccoy}
   Let $\A$ be a computable structure and consider the following three properties:
\begin{enumerate}
   \item\label{akm1} There is a computable function $f$ taking each tuple $\ba$ to a
     computable $\Sinf\alpha$ formula that defines its automorphism orbit.
  \item\label{akm2} $\A$ has a computable $\Pinf{\alpha+1}$ Scott sentence.
  \item\label{akm3} Every tuple in $\A$ has a computable $\Sinf\alpha$-definable automorphism orbit.
\end{enumerate}
\cref{akm1} implies \cref{akm2} implies \cref{akm3}.
\end{theorem}

Effectivizing our results in the previous section, it is not difficult to see
that the Friedman-Stanley embedding preserves \cref{akm1} and \cref{akm3}. 
However, preservation of \cref{akm2} may need a different technique, or not be true at all.
Alvir, Knight and McCoy~\cite{alvir2020a} ask to give a necessary and sufficient
condition for a structure $\A$ to have a computable $\Pinf{\alpha+1}$ Scott sentence.
We offer the following variation closer connected to our studies.

\begin{question}
Does $\A$ have a computable $\Pinf{\alpha+1}$ Scott sentence if and only if
$\T_\A$ has a computable $\Pinf{\alpha+1}$ Scott sentence?
\end{question}

If this question has a positive answer, we suspect that the technique involved
in proving this could be used to answer the question of Alvir, Knight and McCoy.
 \section{$\alpha$-universality and linear orderings}\label{sec:universality}
The main purpose of this section is to show that there is no linear ordering of
Scott sentence complexity $\Sinf{3}$ and that there is no embedding from any
class of structures into linear orderings that is simpler than the
Friedman-Stanley embedding in the sense that it increases the Scott sentence
complexity less.

In order to do this we introduce and study the following notion for the class of
linear orderings.
\begin{definition}\label{def:alphauniversality}
  Let $K_0\subseteq K_1$ be isomorphism invariant classes of structures. The
  class $K_0$ is \emph{$\alpha$-universal} for $K_1$ if for every $\A\in K_1$
  there is $\B\in K_0$ such that $\A\leq_\alpha\B$. 
\end{definition}
\begin{lemma}\label{lem:universalrankinvariant}
  If $K_0$ is $\alpha$-universal for $K_1$, then $K_0$ contains all structures
  in $K_1$ with $\Pinf{\alpha}$ Scott sentences.
\end{lemma}
\begin{proof}
  Suppose otherwise and consider $\A\in K_1\setminus K_0$ with a $\Pinf{\alpha}$
  Scott sentence. By $\alpha$-universality, there is $\B\in K_0$ with
  $\A\leq_\alpha\B$. But then $\A\cong \B$ and thus $\A\in K_0$, a
  contradiction.
\end{proof}
\begin{corollary}
  If $K_0$ is $\alpha$-universal for $K_1$, then $K_0$ contains all strutures in
  $K_1$ of parameterless Scott rank less than $\alpha$.
\end{corollary}
\begin{lemma}\label{lem:universal}
  The class $K=\mathbb N\cup\{\Q\}$ is $2$-universal for $\LO$. In particular,
  $\{\Q\}$ is $2$-universal for infinite linear orderings.
\end{lemma}
\begin{proof}
  Let $\A\in \LO$. If $\A$ is finite then $\A\in K$. If
  $\A$ is infinite then we claim that $\A\leq_2\Q$. As any interval $I$ in
  $\Q$ is infinite we have that $I\leq_1\A_0$ for any interval $\A_0$ of $\A$.
  Thus, $\A\leq_2\Q$ and, hence, $K$ is $2$-universal.
\end{proof}
\subsection{$\Sinf{3}$ is not the Scott sentence complexity of a linear
ordering}\label{sec:notssc}
Alvir, Greenberg, Harrison-Trainor, and Turetsky~\cite{alvir2021} showed that no countable
structure has Scott sentence complexity $\Sinf{2}$ and that no
infinite countable structure has a $\dSinf{1}$ Scott sentence. Thus, the first $\Sigma$ Scott sentence complexity
attainable by a countable structure is $\Sinf{3}$. We are now ready to prove
that there are no linear orderings of this complexity.

Remmel~\cite{remmel1981a} showed that the computably categorical are preciesely
those linear orderings with finite adjacency relation. Since the structures with
$\Sinf{3}$ Scott sentences are precisely those that are relatively
$\bDelta^0_1$ categorical, unsurprisingly, they have a similar
characterization. 
\begin{theorem}\label{thm:remmel}
  A linear ordering has a $\Sinf{3}$ Scott sentence if and only if its adjacency
  relation is finite.
\end{theorem}
\begin{proof}
  Say $\A\in \LO$ has a $\Sinf{3}$ Scott sentence, then by
  \cref{table:invariants} there is a
  parameter $\bp\in A^{<\omega}$ such that $(\A,\bp)$ has a $\Pinf{2}$ Scott
  sentence. By the $2$-universality of $\mathbb N\cup \{\Q\}$ every interval
  induced by $\bp$ is isomorphic to either $\Q$ or $n$. Hence, $\A$ has at
  most finitely many adjacencies.

  On the other hand, say that $Adj^\A$ is finite and let $\bp$ be the ordered
  tuple of elements in the field of $Adj^\A$. Then every interval induced by
  $\bp$ is either empty or isomorphic to $\Q$. As $\Q$ has a $\Pinf{2}$
  Scott sentence, $(\A,\bp)$ has parameterless Scott rank $1$, $\A$ has
  paramaterized Scott rank $2$, and, thus a $\Sinf{3}$ Scott sentence.
\end{proof}
\begin{theorem}\label{thm:nosigma3}
  No linear ordering has Scott sentence complexity $\Sinf{3}$.
\end{theorem}
\begin{proof}
  The only candidates for linear orderings with Scott sentence complexity
  $\Sinf{3}$ are by \cref{thm:remmel} those with finite adjacency relation. Now,
  these linear orderings are either finite, isomorphic to $\Q$, $1+\Q$,
  $\Q+1$ or $1+\Q+1$, or isomorphic to a linear ordering obtained by
  replacing finitely many points in one of the four latter orderings by chains
  of finitely many points. All but the last case are well-known to have
  $\Pinf{3}$ Scott sentences and thus can not have Scott sentence complexity
  $\Sinf{3}$. Consider an ordering of order type in the last case, and that,
  without loss of generality, its order type is $n_1+\Q+n_2+\Q+n_3+\Q+\dots+n_k$. Assume that
  $x$ is the $j$th element in the ordered tuple containing all the successor
  chains and in the block $n_l$. Then its orbit is defined by

  \[\exists x_1\dots x_{j}\forall y\left(y\geq x_1\land x_{j}=x \land \bigwedge_{i<j,i\not\in \{
    n_1,n_1+n_2,\dots,\dots+n_l\}} S(x_i,x_{i+1})\right)\]
  where $S$ is the $\Pi^0_1$ definable successor relation. This definition is
  $\Sinf{2}$ and clearly defines the orbit of elements in the successor chains.
  Now, let $x$ be an element in the $j$th $\Q$ copy where $i=\sum_{l\leq j}
  n_l$ and $N=\sum_{l\leq k} n_l$. Then its automorphism
  orbit is defined by
  \[ \exists x_1\dots x_{N}\left(\bigwedge_{j<N, j\not\in \{
  n_1,n_1+n_2,\dots,\dots+n_l\}} S(x_j,x_{j+1}) \land x_{i}<x<x_{i+1}\right)\]
  and this definition is again $\Sinf{2}$. It is easy to see that the
  defining formula of the automorphism orbit of any ordered tuple is just the
  conjunction of the defining formulas of the automorphism orbits of its
  elements. Thus, $n_1+\Q+n_2+\Q+\dots+n_k$ has a $\Pinf{3}$ Scott sentence.
\end{proof}

\subsection{3-universality and optimality of the Friedman-Stanley embedding}
\begin{theorem}\label{thm:3universal}
The following class is $3$-universal for the class of linear orderings
$K=\{\omega+k, k+\omega^*, k+\Z+k':k,k'\in\mathbb N\}\cup\{\omega+\omega^*\}\cup
\{(\sum_{i\in k} n_i + m_i\cdot\Q )+n_{k+1}:n_i,m_i\in\mathbb N,
n_i>m_{i-1},m_{i+1}\}\cup\{k:k\in\mathbb N\}$. There is no $3$-universal class
$J\subset K$.
\end{theorem}
\begin{proof}
Given a linear order $\L$ we can write $\L=W+\K+R$ with $W$ well ordered or
empty, $R$ reverse well ordered or empty, and $\K$ empty or without a greatest or least element.

The element in $K$ that $\L$ is $\leq_3$-below will depend on the sizes
of $W$, and $R$ and the order type of $\K$. We will deal with these cases
independently.

\emph{Case 1: $W$ is infinite and $R$ is finite.}
In this case, $W=\omega+\alpha$ for some ordinal $\alpha$, so $\L=\omega+\K'+k$
for some $k\in \mathbb{N}$ and $\K'$ is either empty or without greatest element.
We claim that $\L\leq_3\omega+k$.
It is sufficient to show that $\omega+\K'\leq_3\omega$.
The $\exists$-player always wins the corresponding back-and-forth game by
playing the following strategy. In the first round they play isomorphically on
the initial $\omega$. Because of that they only need to win the back-and-forth game for $\omega\leq_2 \omega+\K'$. Say the $\forall$-player picks elements such that
$\omega+\K'$ is partitioned into $n$ intervals $\A_i$, each of cardinality $m_i$ for
$i<n$. The $\exists$-player must pick a partition of $\omega$ such that for all intervals $\B_i$, each of
cardinality $k_i$, $\A_i\leq_1 \B_i$. This inequality holds if and only if
$m_i\geq k_i$. Given this inequality, the intervals formed by picking the first $n$ elements of  $\omega$ 
clearly form the desired partition.

\emph{Case 2: $W$ is finite and $R$ is infinite.} This case is analogous to the
first case.

\emph{Case 3: Both W and R are infinite.}
In this case, we can write $\L=\omega+\K'+\omega^*$.
We claim that $\L\leq_3 \omega+\omega^*$. The winning strategy for the
$\exists$-player is similar to the one for case 1. In the first play they play
isomorphically to reduce the game to showing that $\omega+\omega^*\leq_2
\omega+\K'+\omega^*$. They can then win this game by picking partitions with
smaller intervals, exactly as above.

\emph{Case 3: Both W and R are finite.}
There are two subcases with different behavior here.
In both we assume that $\K$ has no greatest and least element, as if it was
empty, then $\L$ would be finite and thus isomorphic to a structure in $K$.

\emph{Subcase 1: There are arbitrarily large successor chains in $\K$.} Here we
claim that $\K\leq_3\Z$, which gives that $\L\leq_3 k+\Z+k'$ for $k,k'\in\mathbb{N}$.
The $\exists$-player always wins the corresponding game by playing the following
strategy. Assume the $\forall$-player plays an ordered tuple with the distance
between least and greatest element $n$. Then the $\exists$-player plays a
successor chain of size $n$ in $\K$ such that $\K=\K_1+n+\K_2$ where both $\K_1$
and $\K_2$ are infinite. It is then sufficient to show that $\omega^*\leq_2
\K_1$ and $\omega\leq_2 \K_2$. In these games, similar to the cases above, the
$\exists$-player just needs to ensure that they play partitions of smaller size
then the $\forall$-player and this is possible by the same strategy.

\emph{Subcase 2: There is an $n\in\mathbb N$ such that no element in $\K$ has
$n$ successors.} In this case, we consider the block relation $\sim$ given by
$a\sim b$ if and only if $[a,b]$ and $[b,a]$ are finite. By assumption every
block is a linear order of size $n$ or less. As no block can be infinite, no two
blocks can be successors in the quotient $\K/{\sim}$. Furthermore, as there is
no least or greatest element, there are no least or greatest blocks in
$\K/{\sim}$. Therefore, $\K/{\sim}=\Q$. In total, this means that
$\K=\sum_{q\in\Q} i_n(q)$ where $i_n:\Q\to\{1,\cdots, n\}$. If there are only
finitely many blocks of size $n$ , we may write $\K$ as a sum
$\K_1+n+\cdots+n+\K_l$ where no block of size $n$ occurs in any of the $\K_i$.
Repeating this process as necessary, we can write $\K$ as $\K_1+n_1+\K_2+n_2+\dots
\K_l$ where each $\K_i$ is given by $\K_i=\sum_{q\in(b_i,t_i)} i_n(q)$ such
that $\{ q\in (b_i,t_i): i_n(q)=m_i\}$ is infinite for $m_i=\max{\text{range}(i_n\restrict (b_i,t_i))}$.
Furthermore, we have that $m_i>n_i$ and if $i>1$, $m_i>n_{i-1}$.

We now show that $\L=k+\K_1+n_1+\dots
\K_l+k'\leq_3 k+m_1\cdot \Q+n_1\dots m_l\cdot \Q+k'$.
It is enough to show that $\K_i=\sum_{q\in(b_i,t_i)}i_n(q) \leq_3 m_i\cdot\Q$.
The $\exists$-player always wins this game by playing the following strategy.
Every element the $\forall$-player plays in the first play is in a block of size
$m_i$. The $\exists$-player can respond by playing elements in the same position
in blocks of size $m_i$. The resulting intervals in $\K_i$ are of the form
$l_1+\sum_{q\in (b_i^j,t_i^j)}i_n(q)+l_2$ where $b_i^j> b_i$ and $t_i^j<t_i$.
The corresponding intervals in $m_i\cdot \Q$ are of the form $l_1+m_i\cdot
\Q+l_2$. Winning the game thus comes down to finding a winning strategy for the
game $m_i\cdot \Q\leq_2 \sum_{q\in (b_i^j,t_i^j)}i_n(q)$. Given a partition in
$\sum_{q\in (b_i^j,t_i^j)}i_n(q)$, the $\exists$-player needs to find a
partition in $m_i\cdot \Q$ such that all intervals are smaller than the intervals
in the partition picked by the $\forall$-player. They can do
that using the fact that the finite blocks are densely ordered and always of maximum size.

This shows that $K$ is $3$-universal. Standard arguments show that each
element of $K$ has a $\Pinf3$ Scott sentence. Therefore, by \cref{lem:universalrankinvariant},
there is no $J\subset K$ such that $J$ is $3$-universal.
\end{proof}
Note that McCoy~\cite{mccoy2003} obtained a similar characterization for the linear orderings that are
$\Delta^0_2$ categorical. \cref{thm:3universal} can be used to show a boldface
version of this result.
\begin{corollary}\label{cor:pinf3ss}
  If $\L$ is a linear ordering with a $\Pinf{3}$ Scott sentence, then $\L$ is in
  $K$. 
\end{corollary}
\begin{corollary}\label{cor:sinf4ss}
  If $\L$ is a linear ordering with a $\Sinf{4}$ Scott sentence, then $\L$ is a
  finite sum of linear orderings in $K$.
\end{corollary}
That the Friedman-Stanley embedding is optimal for Scott sentence
complexity now follows easily from \cref{cor:sinf4ss} by cardinality
considerations.
\begin{theorem}\label{thm:optimalembedding}
   There is no Turing computable embedding $\Phi:Mod(\tau)
   \to \LO$ such that for all $\A\in Mod(\tau)$, $SSC(\A)=\Pinf{2}$ if and only
   if $SSC(\Phi(\A))\in\{ \Sinf{4}, \dSinf{3}, \Pinf{3}\}$ for any countable
   vocabulary $\tau$.
\end{theorem}
\begin{proof}
   We may assume that $\tau$ contains only one binary relation symbol, i.e.,
   that the $\tau$-structures are graphs. Notice that there are continuum many
   graphs with Scott sentence complexity $\Pinf{2}$: For any $X\subseteq \omega$
   the graph consisting of a disconnected cycle of size $n+2$ for every $n\in X$ has a
   $\Pinf{2}$ Scott sentence. However, by \cref{cor:sinf4ss} there
   are only countably many linear orderings with a $\Sinf{4}$ Scott sentence. As any potential
   embedding $\Phi$ needs to respect isomorphism, it cannot map every graph
   with a $\Pinf{2}$ Scott sentence to a linear ordering with a Scott sentence
   of complexity $\{\Sinf{4}, \dSinf{3},\Pinf{3}\}$.
\end{proof}
\begin{proposition}
  There is no countable class of $4$-universal linear orderings.
\end{proposition}
\begin{proof}
  As mentioned in the proof of \cref{thm:optimalembedding} there are uncountably
  many graphs with Scott sentence complexity $\Pinf{2}$. The Friedman-Stanley
  embedding transforms those into linear orderings of Scott sentence complexity
  $\Pinf{4}$. Thus, by \cref{lem:universalrankinvariant} there can not be a countable $4$-universal set of
  linear orderings.
\end{proof}
 \section{Possible Scott sentence complexities of linear orderings}\label{sec:ssc}
We now exhibit possible Scott sentence complexities of linear orderings. Many of
our examples are constructed by producing shuffle sums of linear orderings.
Recall that given a set $S$ of countable linear orderings, the \emph{shuffle
sum} $Sh(S)$ of $S$ is obtained by partitioning $\Q$ into $|S|$ mutually dense
sets $K_i$ and replacing each point in $K_i$ with a copy of $S_i$.
Note that a linear order of the form $L\cdot\Q$ is a special case of a shuffle
sum. We begin by proving a couple of lemmas about the back-and-forth relations
of shuffle sums.

\begin{lemma}\label{lem:shufflesum}
Let $\A$ be a shuffle sum and $\B = \L\cdot\Q$ for some ordering $\L$. If
$\A\leq_\alpha \B$ then $\A+\L+\B\leq_{\alpha+2}\A+\B$.
\end{lemma}
\begin{proof}
To show this, we describe a winning strategy for the $\exists$-player in the
$\alpha+2$-back-and-forth game. 
On the first turn, the $\exists$-player copies the selected tuple of the
$\forall$-player using an isomorphism of the initial $\A$s and the final $\B$s. 
The only possible non-isomorphic interval is the one between the smallest
element $b$ in $\B$ and the largest element $a$ in $\A$. 
In other words, we need to show that 
\[\A_{>a}+\L+\B_{<b}\geq_{\alpha+1} \A_{>a}+\B_{<b}.\]
Because $\A$ and $\B$ are shuffle sums, $\A_{>a}\cong N+\A$ and $\B_{<b}\cong
\B+K$ for some linear orderings $N$ and $K$. 
This means that we only need to show that
$\A+\L+\B\geq_{\alpha+1}\A+\B$
if we play isomorphically on $N$ and $K$.

On the next turn, the $\forall$-player plays on $\A+\L+\B$, and the
$\exists$-player responds on $\A+\B$ using an isomorphism between the initial
$\A$s and an isomorphism of $\L+\B$ to a final segment of $\B$.
Assume without loss of generality, that $c$ is the smallest element in $\L$
played by the $\forall$-player in the distinguished copy of $\L$ within $\A+\L+\B$.
The only possible non-isomorphic interval is between $c$ and the
largest element $a$ selected within $\A$.
In other words, we need to show that 
\[\A_{>a}+\L_{<c}\leq_{\alpha} \A_{>a}+\B+\L_{<c}.\]
As $\A_{>a}\cong M+\A$ and the $\exists$-player can play the isomorphism between
the copies of $\L_{<c}$ and $M$, it is sufficient to show that $\A\leq_{\alpha}\A+\B$.
As $\A$ is a shuffle sum $\A\cong \A+\A$.
Therefore, it is enough to show that $\A\leq_\alpha \B$, as desired.
\end{proof}

\begin{corollary}\label{lower}
Let $\A$ be a shuffle sum and $\B = \L\cdot\Q$ for some ordering $\L$. If $\A+\L+\B\not\cong \A+\B$ and
$\A\leq_\alpha \B$, then $\A+\L+\B$ has no $\Pinf{\alpha+2}$ Scott sentence.
\end{corollary}

We are now ready to provide the constructions of example linear orders.

\begin{theorem}
There is a linear ordering with Scott sentence complexity $\Sinf{4}$.
\end{theorem}

\begin{proof}
Consider the linear order $\L=2\cdot\Q+1+\Q$ and call the $1$ in the middle $c$.
We claim that that $\SR(\L,c)=\max\{\SR(2\cdot\Q),\SR(\Q)\}\leq2$ (the first equality
follows from~\cite[Lemma 11]{GM23}). The inequality follows from the fact that
in $2\cdot\Q$ the orbits of singletons are defined by the $\Sinf{2}$-formulas $\exists y\
S(x)=y$ and $\exists y\ S(y)=x$. This definition extends to tuples by adding in
the order of the tuple to the definition along with any successor relations that hold.

We now show that $\L$ has Scott sentence complexity $\Sinf4$.
This will be done by showing that it does not have a $\Pinf4$ Scott sentence. 
In order to do this, we appeal to \cref{lower} which implies that all we need to show is that $2\cdot\Q\leq_2\Q$.
However, this follows immediately from the 2-universality of $\Q$ among infinite
linear orders (\cref{lem:universal}).
\end{proof}

Note that the above proof also implies that $2\cdot\Q+1+\Q\leq_42\cdot\Q+\Q$.

\begin{theorem}
There is a linear order with Scott sentence complexity $\Sigma_5^{in}$.
\end{theorem}

\begin{proof}

Let $\A=Sh(\{1,\omega\})$, $\B=\omega\cdot\Q$, $\L=\A+\omega+\B$ and denote the
first element of the $\omega$ in the middle $c$. We claim that $\L$ has Scott
sentence complexity $\Sinf{5}$. 

First, we show that $\SR(\L,c)=\max\{\SR(\A),\SR(\B)\}\leq3$ (the first equality
follows from~\cite[Lemma 11]{GM23}). 
Define the following family of formulas that denote if $x$ is the top of an
successor chain of length $n>1$:
\[S_n(x):\hspace{.5 in} \exists x_1,\cdots,x_{n-1} \bigwedge_{i<n-1}
S(x_i,x_{i+1}) \land S(x_{n-1},x).\]
Note that $S_n(x)$ is $\Sinf{2}$. It is not difficult to see that
the automorphism orbits of elements in $\B$ are given by the formulas
\[\forall y \neg S(y,x) \text{ and } S_n(x)\land\lnot S_{n+1}(x) \text{ for }
n\in\omega.\]
These formulas are all $\Sinf 3$ and the automorphism orbits of tuples can be
easily constructed by taking conjunctions over them. Thus, $\SR(\B)\leq 3$.
Similarly, within $\A$ the automorphism orbits of elements are given by the
formulas
\[\forall y\  \neg S(x,y),\ \forall y\ \neg S(y,x) \land \exists y\ S(x,y)\
\text{and} \ S_n(x)\land\lnot S_{n+1}(x) \text{ for all $n\in\omega$}.\]
To extend these definitions to tuples we only need to add the successor relations that hold between the elements and the order of the elements.
All of these definitions can be given by $\Sigma_3^{in}$ and therefore, $\SR(\A)\leq 3$.

It remains to show that $\L$ does not have a $\Pinf{5}$ Scott sentence. 
By \cref{lower} it is enough to show that $\omega\cdot\Q\geq_3\A$.
Towards this let $\bar{p}\in \omega\cdot\Q$ be an ordered tuple. We want to find
$\bq\in \A$ such that $(\omega\cdot \Q,\bp)\leq_2(\A,\bq)$. We may assume
without loss of generality that $\bp$ is ordered and contains no elements from
the same block. Now find an
ordered tuple $\bar{q}\in \A$ such that every $q_i$ is in an $\omega$-block and
in the same position within that block as $p_i$. Note that, for each
$i$ there is some $k_i\in\omega$, such that $(p_i,p_{i+1})\cong \omega+\A+k_i$
and $(q_i,q_{i+1})\cong \omega+\omega\cdot\Q+k_i$. Therefore, it is enough to show that
\[\A\geq_2 \omega\cdot\Q.\]

Repeating the style of argument from above we fix a tuple $\bar{p}\in \A$ and
find a suitable $\bar{q}\in \omega\cdot\Q$ such that
$(\A,\bar{p})\leq_1(\omega\cdot\Q,\bar{q}).$ Recall that $\A\leq_1 \B$ if and
only if $|\B|\leq |\A|$. Assuming that $\bar{p}$ is ordered, $(p_i,p_{i+1})$ has
cardinality $n\in\omega$ or $\aleph_0$ with the first and last intervals always
having cardinality $\aleph_0$. Any combination of gaps of size $n$ and
$\aleph_0$ can also be found in $\omega\cdot\Q$ and we can thus find a suitable
tuple $\bar{q}$ to finish the proof.
\end{proof}
Given these base case examples of linear orders of small Scott sentence complexity, we now devise a method of systematically using these examples to fill in most of the larger Scott sentence complexities. 
The examples are quite simple; we consider orders of the form $\Z^\alpha\cdot
\L$ where $\L$ is one of our previously constructed examples. 
The more difficult work is determining exactly of what Scott sentence complexity these orders are. 
 
This will be proven by demonstrating an upper and lower bound on the complexity of the Scott sentence. 
We begin with proving the lower bound. 
In the proof of the below lemma we use $\zeta_\alpha$ to denote the unique up to
isomorphism initial segment of $\Z^\alpha$. 
It is easy to check that $\zeta_\alpha^*$ is the unique end segment of  $\Z^\alpha$ and that $\zeta_{\alpha+1}^*=\zeta_\alpha+\Z^\alpha\cdot\omega$. 

\begin{lemma}\label{lem:timesZ}
For the sake of organization, consider the following ordinal indexed propositions.
\begin{enumerate}\tightlist
\item[($A_\alpha$)] For any $\K$ and $\L$, $\Z^\alpha\cdot \K
    \leq_{2\alpha}\Z^\alpha\cdot \L$.
\item[($B_\alpha$)] For any $\K$ and $\L$ with $|\K|\geq|L|$, $\Z^\alpha\cdot \K
    \leq_{2\alpha+1}\Z^\alpha\cdot \L$.
\item[($C_\alpha$)] For any $\K$ and any $\L$ without a last element,
    $\Z^\alpha\cdot(\omega+\K)\leq_{2\alpha+2} \Z^\alpha\cdot(\omega+\L)$.
\end{enumerate}
For any countable ordinal $\alpha$, $A_\alpha$, $B_\alpha$ and $C_\alpha$ are true.
\end{lemma}

\begin{proof}
    The proof is by induction with the statements $A_0$ and $B_0$ trivially true.
    We start by showing that $B_\alpha$ implies $A_{\alpha+1}$. 

    Select a tuple $a_0\cdots a_n$ in $\Z^{\alpha+1}\cdot \L=\Z^\alpha\cdot\Z\cdot
    \L$. We will find a tuple $b_0\dots b_n$ in $\Z^{\alpha+1}\cdot
    \K=\Z^\alpha\cdot\Z\cdot \K$ such that $(\Z^{\alpha+1}\cdot \L,\ba
    )\leq_{2\alpha+1} (\Z^{\alpha+1}\cdot \K,\bb)$. Pick a point $b_0$ arbitrarily.
    Given $b_j$, we pick $b_{j+1}$ as follows. If $a_j$ and $a_{j+1}$ are in the
    same copy of $\Z^{\alpha+1}$, we can pick $b_{j+1}$ such that
    $(a_j,a_{j+1})\cong(b_j,b_{j+1})$. If not, we pick $b_{j+1}$
    arbitrarily in the copy of $\Z^{\alpha}$ that is the successor of the successor
    of the copy of $b_j$. In other words, we pick $b_{j+1}$ with  $(b_j,b_{j+1})\cong \zeta^*_\alpha +\Z^\alpha+\zeta_\alpha$.
    
To show that our choice of $\bb$ is correct, we have to confirm the following
relations between non-isomorphic intervals.
\begin{enumerate}\tightlist
\item $(b_j,b_{j+1})\cong \zeta^*_\alpha +\Z^\alpha+\zeta_\alpha \geq_{2\alpha+1} \zeta^*_\alpha
    +\Z^\alpha\cdot \L_1+\zeta_\alpha$, where $\L_1$ is an infinite order,
\item $(-\infty,b_0)\cong \Z^\alpha\cdot \K_1+\zeta_\alpha \geq_{2\alpha+1}\Z^\alpha\cdot
    \L_1+\zeta_\alpha$, where $\L_1$ and $\K_1$ are some initial segments of the $\Z^\alpha$, and therefore infinite, or
\item $(b_{n},\infty)\cong\zeta^*_\alpha +\Z^\alpha\cdot \L_1 \geq_{2\alpha+1} \zeta^*_\alpha
    +\Z^\alpha\cdot \K_1$, where $\L_1$ and $\K_1$ are some final segments of the $\Z^\alpha$, and therefore infinite.
\end{enumerate} 
It follows immediately from $B_\alpha$ that all of these back-and-forth relations hold.

We now show that $B_\alpha$ implies $C_\alpha$. So assume that $\K$ and $\L$ have
no last element and that $B_\alpha$ holds. 
Select a tuple $a_0\cdots  a_n$ in $\Z^\alpha\cdot(\omega+\L)$.
Without loss of generality, $a_0$ is in the first copy of $\Z^\alpha$. We will
find a tuple $b_0\dots b_n$ in $\Z^\alpha\cdot(\omega+\K)$ such that
$(\Z^\alpha\cdot (\omega+\L),\ba)\leq_{2\alpha+1}
(\Z^\alpha\cdot(\omega+\K),\bb)$. For the $a_j$ in the initial $\Z^\alpha\cdot\omega$, the $b_j$ are picked isomorphically. 
We now define the rest of the $b_j$ by induction so that all of them are in the initial $\Z^\alpha\cdot\omega$.
If $a_j$ is in the same $\Z^\alpha$ block as $a_{j+1}$ or $a_{j+1}$ is in the
succesor of $a_j$'s block, define $b_{j+1}$ so that
$(b_j,b_{j+1})\cong (a_j,a_{j+1})$.
Otherwise, define $b_{j+1}$ as being some element in the successor of the
successor of the $\Z^\alpha$ block of $b_j$. 
Not including isomorphic intervals, we only need to check the following cases. 
\begin{enumerate}
    \item $(b_j,b_{j+1})\cong\zeta^*_\alpha +\Z^\alpha+\zeta_\alpha \geq_{2\alpha+1} \zeta^*_\alpha
    +\Z^\alpha\cdot \L_1+\zeta_\alpha$ where $\L_1$ is infinite, 
\item $(b_n,\infty)\cong\zeta^*_\alpha +\Z^\alpha\cdot (\omega+\K) \geq_{2\alpha+1}
    \zeta^*_\alpha +\Z^\alpha\cdot \L_1$, where $\L_1$ is some final segment of
    $\omega+\L$, and thus infinite.
\end{enumerate} 
All of these relations follow from $B_\alpha$, as desired.

To complete the proof of the successor step of the induction, we show that
$C_\alpha$ and $A_{\alpha+1}$ imply $B_{\alpha+1}$.

Select a tuple  $a_0\cdots a_n$ in $\Z^{\alpha+1}\cdot \L$. These points belong
to at most $n$ distinct copies of $\Z^{\alpha+1}$. Because $|\K|\geq| \L|$, we
can match each chosen copy of $\Z^{\alpha+1}$ to one in $\Z^{\alpha+1}\cdot \K$.
Furthermore, if there is a copy of $\Z^{\alpha+1}$ between two of the chosen
copies in $\Z^{\alpha+1}\cdot \L$, we can assure that the corresponding copies
in $\Z^{\alpha+1}\cdot \K$ are also not successors. Also, if $a_n$ is not in the
last copy of $\Z^{\alpha+1}$, we can assure that neither is $b_n$ and, similarly
for $a_0$ and $b_0$ in the first copy. In particular, not including the
isomorphic intervals, we only need to check the following cases.
\begin{enumerate}
    \item $(b_j,b_{j+1})\cong \zeta^*_{\alpha+1} +\Z^{\alpha+1}\cdot
        \K_1+\zeta_{\alpha+1}\geq_{2\alpha+2}\zeta^*_{\alpha+1}
        +\Z^{\alpha+1}\cdot \L_1+\zeta_{\alpha+1}$ for some orders $\K_1$ and $\L_1$,
    \item $(b_n,\infty)\cong\zeta^*_{\alpha+1} +\Z^{\alpha+1}\cdot
        \K_1\geq_{2\alpha+2}\zeta^*_{\alpha+1} +\Z^{\alpha+1}\cdot \L_1$ for
        some orders $\K_1$ and $\L_1$,
\item $(-\infty,b_0)\cong\Z^{\alpha+1}\cdot
    \K_1+\zeta_{\alpha+1}\geq_{2\alpha+2}\Z^{\alpha+1}\cdot
    \L_1+\zeta_{\alpha+1}$ for some orders $\K_1$ and $\L_1$,
\item $(-\infty,b_0)\cong\zeta_{\alpha+1}\geq_{2\alpha+2} \Z^{\alpha+1}\cdot
    \L_1+\zeta_{\alpha+1}$ for some order $\L_1$,
\item $(b_n,\infty)\cong\zeta^*_{\alpha+1} +\Z^{\alpha+1}\cdot
    \K_1\geq_{2\alpha+2} \zeta^*_{\alpha+1} $ for some order $\K_1$.
\end{enumerate}

The first three cases are handled immediately by $A_{\alpha+1}$. By symmetry it
is sufficient to show case (5). However, this is the same as showing that 
$$\zeta_{\alpha}^* + \Z^\alpha\cdot\omega \leq_{2\alpha+2}\zeta_{\alpha}^* +
\Z^{\alpha}\cdot (\omega+\Z\cdot \K_1).$$
This follows directly from $C_\alpha$.

At last we consider the remaining limit levels. We already have seen that
$B_\lambda$ implies $C_\lambda$. The statement $A_\lambda$ follows immediately from $A_\gamma$ for $\gamma<\lambda$. 
A bit more nuanced is the issue of $B_\lambda$. It follows from $A_\lambda$ and $C_\gamma$ for $\gamma<\lambda$.
Analogous to the proof in the successor case, we analyze the possibilities:
\begin{enumerate}
\item $\zeta^*_{\lambda} +\Z^{\lambda}\cdot
    \K_1+\zeta_{\lambda}\geq_\lambda \zeta^*_{\lambda} +\Z^{\lambda}\cdot \L_1+\zeta_{\lambda} $ for some orders $\L_1$ and $\K_1$,
\item $\zeta^*_{\lambda} +\Z^{\lambda}\cdot \K_1\geq_\lambda \zeta^*_{\lambda} +\Z^{\lambda}\cdot \L_1 $ for some orders $\L_1$ and $\K_1$,
\item $\Z^{\lambda}\cdot \K_1+\zeta_{\lambda}\geq_\lambda \Z^{\lambda}\cdot \L_1+\zeta_{\lambda} $ for some orders $\L_1$ and $\K_1$,
\item $\Z^{\lambda}\cdot \K_1+\zeta_{\lambda}\geq_\lambda \zeta_{\lambda} $ for some order $\K_1$,
\item $\zeta^*_{\lambda} +\Z^{\lambda}\cdot \K_1\geq_\lambda\zeta^*_{\lambda} $ for some order $\K_1$.
\end{enumerate}

The first three cases are handled immediately by $A_{\lambda}$. By symmetry it is clear that it is enough to show case 5. However, this is the same as showing that for all $\gamma<\lambda$,
$$\zeta_{\gamma}^* + \Z^\gamma\cdot(\omega +  \Z\cdot\zeta_{\lambda}^*)
\leq_{\gamma}\zeta_{\gamma}^* + \Z^{\gamma}\cdot
(\Z\cdot\zeta_{\lambda}^*+\Z^\lambda\cdot \K_1),$$
which follows from the $C_\gamma$ for $\gamma<\lambda$.

Therefore, for any countable ordinal $\alpha$, $A_\alpha$, $B_\alpha$ and $C_\alpha$ all hold.
\end{proof}

The proposition $B_\alpha$ is analogous to Lemma II.38 in \cite{montalban2021}.
However, it has the notable advantage of being more universal, cleaner to state
and easier to apply for our purposes. It is also independently interesting as it
furthers our understanding of the behavior of powers of $\Z$, which have been
studied in \cite{GHKMMS} and \cite{Ash91}. For our purpose, what is important is
that the proposition $B_\alpha$ serves as the base case for the following lemma,
which generalizes Lemmas 7.2 and 7.3 from \cite{GHKMMS} and proposition 4.8 from \cite{Ash91}.
\begin{lemma}\label{lem:zpowersbfpreserve}
For any $\L$ and $\K$
$$\L\leq_\beta \K\implies \Z^\alpha\cdot \L\leq_{2\alpha+\beta} \Z^\alpha\cdot \K.$$
\end{lemma}

\begin{proof}
 We demonstrate this by induction on $\beta$. Notice that the base case is given by the above observation. 
 The limit case follows immediately by the definition of the back-and-forth relations along with the observation that for any linear order $\N$, $\L\leq_\beta \K \implies \N\cdot \L\leq_\beta \N\cdot \K$.
 Thus, we need only look at the successor case.

 Let $\beta=\gamma+1$ and assume $\L\leq_{\gamma+1} \K$. We need to show that
 $\Z^\alpha\cdot \L\leq_{2\alpha+\gamma+1}\Z^\alpha\cdot \K.$ We view
 $\Z^{\alpha}\cdot K$ as a product ordering with elements of the the form
 $(\delta,b)$ where $b\in K$ and $\delta\in \Z^{\alpha}$. Consider a play of the
 game where the $\forall$-player plays a tuple $(\delta_i,b_i)_{i\in k}$ from $\Z^\alpha\cdot
 \K$. If $s$ is a winning strategy for demonstrating that $\L\leq_{\gamma+1}
 \K$, the $\exists$-player may play $(\delta_i,s(b_i))_{i\in k}$ in response. Let
 $b_{-1}=s(b_{-1})=-\infty$ and $b_k=s(b_k)=\infty$ by convention. We need only
 show that, $$[(\delta_i,b_i),(\delta_{i+1},b_{i+1})] \leq_{2\alpha+\gamma}
 [(\delta_i,s(b_i)),(\delta_{i+1},s(b_{i+1}))].$$ Note that the left hand side
 is isomorphic to $\zeta_\alpha+\Z^\alpha\cdot(b_i,b_{i+1})+\zeta_\alpha^*$ and
 the right hand side is isomorphic to
 $\zeta_\alpha+\Z^\alpha\cdot(s(b_i),s(b_{i+1}))+\zeta_\alpha^*$. As $s$ is a
 winning strategy, by induction,
 \[\Z^\alpha\cdot(b_i,b_{i+1})\leq_{2\alpha+\gamma}
 \Z^\alpha\cdot(s(b_i),s(b_{i+1})),\] which demonstrates the claim.
\end{proof}

The following Proposition follows immediately from \cref{lem:zpowersbfpreserve}.
\begin{proposition}
For any linear order $\L$, if there is a $\K$ with $\L\leq_\beta \K$ yet $\L\not\cong \K$ then $\Z^\alpha\cdot \L$ has no $\Pinf{2\alpha+\beta}$ Scott sentence.
For any linear order $\L$, if there is a $\K$ with $\L\geq_\beta \K$ yet $\L\not\cong \K$ then $\Z^\alpha\cdot \L$ has no $\Sinf{2\alpha+\beta}$ Scott sentence.
\end{proposition}

This allows to transfer lower bounds for the Scott sentence complexity of $\L$ to $\Z^\alpha\cdot \L$.

We now move to proving the upper bound. In particular, we must demonstrate that
multiplying by a power of $\Z$ does not make the Scott rank too high. For this
we use the notion of the block relation $\sim_\alpha$. We follow the definitions
in~\cite{AR20}.

\begin{lemma}
If $\L$ has a Scott sentence of complexity $\Ginf{\beta}$ for
$\Gamma\in\{\Sigma,\Pi,d\text{-}\Sigma\}$, then the linear order $\Z^\alpha\cdot \L$ has a Scott sentence of complexity $\Ginf{2\alpha+\beta}$.
\end{lemma}

\begin{proof}
Given a formula $\phi$ in the language of linear orders we define $\phi_\alpha$
as being the same as $\phi$ except for the fact that instances of $x<y$ are
replaced by $x<y \land x\not\sim_\alpha y$ and instances of $x=y$ are replaced
by $x\sim_\alpha y$. Note that if $\phi$ is $\Ginf{\beta}$, then $\phi_\alpha$
is at worst $\Ginf{2\alpha+\beta}$ by the fact that $\sim_\alpha$ is
$\Sinf{2\alpha}$ definable (see~\cite[Proposition 4]{AR20}). It is also not
difficult to see that if $\phi$ is a Scott sentence for $\L$, then
$\N\models\phi_\alpha$ guarantees that $\N{/}{\sim_\alpha}\cong \L$. 
Finally, we define
\[\psi= ~ \forall x \bigwwedge_{\delta<\alpha} \bigwwedge_{n\in\omega} (\exists
y\ S_\delta^n(x)=y)\land (\exists y\ S_\delta^n(y)=x).\]

Here, $S_\delta^n(x)=y$ is shorthand for saying that $y$ is the $n$th
$\delta$-successor of $x$. In other words,
\[ \exists z_0\cdots z_{n} ~ x=z_0
\land y=z_n \land \bigwedge_{i<n} z_i\not\sim_\delta z_{i+1} \land \forall w~
z_i<w<z_{i+1} \to (w\sim_\delta z_i \lor w\sim_\delta z_{i+1}).\] 
Overall, this is $\Sinf{2\delta+2}$. This
gives that $\psi$ is, at worst, $\Pinf{2\alpha+1}$. Furthermore, $\psi$
guarantees that every $\sim_\alpha$ equivalence class is isomorphic to
$\Z^\alpha$. Therefore, $\phi_\alpha\land\psi$ gives the desired Scott sentence
for $\Z^\alpha\cdot \L$.
\end{proof}

This result along with the previous constructions gives that there are linear orders with any Scott sentence complexity that is not too close to a limit ordinal.

\begin{corollary}\label{cor:SSCsuccessors}
There are linear orders of the following Scott sentence complexities:
\begin{enumerate}
\item $\Sinf{\alpha+n}$ for any countable ordinal $\alpha$ and $n\geq 4$,
\item $\dSinf{\alpha+n}$ for any countable ordinal $\alpha $ and $n\geq 1$,
\item $\Pinf{\alpha+n}$ for any countable ordinal $\alpha$ and $n\geq 1$.
\end{enumerate}
\end{corollary}

\begin{proof}
Let $\L_4$ be the constructed example of Scott sentence complexity $\Sinf4$ and $\K_4$ such that $\L_4\leq_4 \K_4$. 
Similarly define $\L_5$ and $\K_5$.
For any $\alpha$ we know that $\Z^\alpha\cdot \L_4$ has a $\Sinf{2\alpha+4}$
Scott sentence, yet it has no $\Pinf{2\alpha+4}$ Scott sentence as
$\Z^\alpha\cdot \L_4\leq_{2\alpha+4} \Z^\alpha\cdot \K_4$. Thus, $\Z^\alpha\cdot
\L_4$ has Scott sentence complexity $\Sinf{2\alpha+4}$. Similarly, $\Z^\alpha\cdot \L_5$ has Scott sentence complexity $\Sinf{2\alpha+5}$. These constructions provide examples for all of the claimed cases of the form $\Sinf{\alpha+n}$.

Let $\L_2:=\Q+2+\Q$. It is not difficult to see that this order has a $\dSinf2$
Scott sentence and that $\Q+3+\Q\leq_2 \L_2\leq_2 \Q$, so this is indeed
optimal. Using the same reasoning as above, we see that $\Z^\alpha\cdot \L_2$
has Scott sentence complexity $\dSinf{2\alpha+2}$. Ash~\cite{ash1986} analyzed
the back-and-forth relations of well-orderings and it follows from his results
that $\omega^\alpha\cdot 2$ has Scott sentence complexity $\dSinf{2\alpha+1}$
(see~\cite[proof of Proposition 19]{AR20}). These constructions provide examples
for all of the claimed cases of the form $\dSinf{\alpha+n}$.

It is a basic exercise to show that $\Q$ has Scott sentence complexity $\Pinf2$.
Using the same reasoning as above, we see that $\Z^\alpha\cdot \Q$ has Scott
sentence complexity $\Pinf{2\alpha+2}$. Furthermore, it again follows from
results of Ash~\cite{ash1986} that $\omega^\alpha$ has Scott sentence complexity
$\Pinf{2\alpha+1}$. These constructions provide examples for all of the claimed
cases of the form $\Pinf{\alpha+n}$. \end{proof}

This covers every case except for some possibilities that lie close to limit
ordinals. The following results fill most of those gaps.

\begin{proposition}\label{prop:scatteredlimit}
For any limit ordinal $\lambda,$ there is a scattered linear order of Scott sentence complexity $\Pinf{\lambda}.$
\end{proposition}

\begin{proof}
Let $(\delta_n)_{n\in\omega}$ be a fundamental sequence for $\lambda$. We show that
\[\L_\lambda:=\sum_{i\in\omega}i+\Z^{\delta_n}\]
is of the desired complexity. 

First note that it indeed does have a $\Pinf{\lambda}$ Scott sentence. The sentence states the following:
\begin{enumerate}\tightlist
    \item There is exactly one 1-block isomorphic to each natural number. 
    \item These 1-blocks are ordered like the natural numbers.
    \item The order between $n$ and $n+1$ is isomorphic to $\Z^{\delta^n}.$
\end{enumerate}
This is $\Pinf{\lambda}$ because each of the items is $\Pinf{<\lambda}$. 
The first two statements can be expressed using finitely many alternations of quantifiers. 
The final statement can be expressed by the Scott sentence of $\Z^{\delta_n}$
relativized to the specific interval, this sentence has Scott sentence
complexity below $\lambda$ by~\cite{AR20}, as the Hausdorff rank of the interval
is below $\lambda$. 

By Lemma~\ref{lem:timesZ}, the $\Z^{\delta_n}$ have unbounded Scott rank below $\lambda$. 
As these are all $\Dinf0$-definable over parameters in $\L_\lambda$,
by~\cite[Lemma 4.3]{MonIntermediate}, we have that for each $n$, $\SR(\L_\lambda)\geq\delta_n,$ so $\Pinf{\lambda}$ is indeed the optimal Scott sentence complexity.
\end{proof}

We can use a variation of this construction to get a structure of Scott sentence complexity $\Sinf{\lambda+2}.$
Unlike some of the previous constructions, this construction is not a sum of shuffle sums, so Corollary \ref{lower} will not apply.

\begin{theorem}\label{thm:Slim+2}
For any limit ordinal $\lambda,$ there is a linear order of Scott sentence
complexity $\Sinf{\lambda+2}.$
\end{theorem}
\begin{proof}
    Let $(\delta_n)_{n\in\omega}$ be a
    fundamental sequence for $\lambda$ and let 
\[\L_\lambda:=\sum_{i\in\omega}i+\Z^{\delta_n}.\]
Furthermore, define for every $n$, the linear orderings:
\[\L_{\lambda,n}:=\sum_{i< n} i+\Z^{\delta_i}+\sum_{i\geq n} i+\Z^{\delta^n}.\]
Note that by Lemma \ref{lem:timesZ}, $\L_\lambda$ is approximated by the
sequence of these orderings in the sense that $\L_\lambda\equiv_{2\delta_n}\L_{\lambda,n}.$

Let $\B=\L_\lambda\cdot(1+\Q)$ and $\A$ be a copy of $\Q$ where an unbounded,
increasing sequence of points $c_n$ is replaced by $\L_{\lambda,n}$ and each
other point is replaced by $\L_\lambda$. We now claim that $\K=\A+\B$ is the
desired order. Note that over the parameter of the initial element of $\B$,
$\SR_p(\K)=\max(\SR(\A),\SR(\B))$ (this equality follows from~\cite[Lemma
11]{GM23}). We now demonstrate that this quantity is exactly $\lambda$.

To define the automorphism orbits of elements within $\B$ we have to distinguish
elements in the initial copy of $\L_\lambda$. For an element $b\in \B$, let $n$
be the size of the finite block to the left of $b$. We can write a formula of
finite rank saying that there is precisely one finite block of size $n$
to the left of $b$ and thus use this type of formulas to distinguish elements in
the first $\L_\lambda$ copy. If $b$ is in a finite block we conjunct this
formula or its negation with its position in this finite block, if $b$ is in a
copy of $\Z^{\delta_n}$, then we relativize the formula defining its
automorphism orbit in this ordering to the interval restricted by block of size
$n$ to the left and a block of size $n+1$ to the right, just as in the proof of
\cref{prop:scatteredlimit}. Conjuncting this with the formula deciding whether
$b$ is in the initial copy of $\L_\lambda$ yields the defining formula for its
automorphism orbit. To extend this definition to tuples we take the conjunction
of the formulas defining orbits for elements and a formula specifying the order
of the tuple.

For $\A$ the definitions are a little more subtle, but generally follow the same
plan. In particular, if the point in $\Q$ that $x$ lies in is between $c_n$ and
$c_{n+1}$, then the following describes the automorphism orbit of $x$:

\begin{enumerate}
    \item Let $z$ be the greatest element below $x$ with no successor or predecessor.

    \item Let $z_n$ be the initial element with no successor or predecessor
        within $c_n$; i.e. say that it is the greatest element with no successor
        or predecessor below an $n+1$ and $n+2$ block that bound a structure
        isomorphic to $\Z^{\delta_n}$.

    \item Define $z_{n+1}$ analogously to $z_n$, expect it is in $c_{n+1}.$
    \item Say that $z_n\leq x<z_{n+1}$.
    \item Check if $z=z_n$ or not.

    \item In an analogous manner to the case of $\A$, define the automorphism
        orbit of $x$ within its copy of $\Z^{\delta_k}$ or the finite
        block that it is in.

\end{enumerate}

This describes the automorphism orbit of a point and is of complexity less than
$\Sinf{\lambda}$. In order to extend this definition to tuples the only
additional detail needed is the order of the tuple and whether the tuples lie in the same point in $\Q$ or not. This can be done by comparing the various values of "$z$" that emerge for each point in the tuple.

In order to finish the proof, all that remains is to show that
$\SR(\K)=\lambda+2$, or in other words, that it has an orbit that it is not
$\Sinf{\lambda+1}$-definable. Let $c$ be the initial point in $\B$ and $b$ be
some other point in $\B$ in a block of size $1$. We will show that 
$(\K,b)\leq_{\lambda+1}(\K,c)$, so the orbit of $c$ is not $\Sinf{\lambda+1}$
definable in $\K$.

Note that $\K_{>c}\cong \K_{>b}$, so it is sufficient to show that
$\K_{<b}\leq_{\lambda+1} \K_{<c}.$
As $\K_{<c}$ is an initial segment of $\K_{<b}$, for the first turn of the game,
the $\exists$-player can pick whatever points the $\forall$-player picked in
$\K_{<c}$ according to the canonical embedding of $\K_{<c}$ in $\K_{<b}$.
Note that each interval is isomorphic except for the one between the greatest chosen points (call them $a_k$) and $c$ or $b$ respectively. 
$a_k$ may be greater than some finite number of the $c_n$. 
Therefore, up to a finite difference in the choice of sequence $\delta_n$ used
to define $\L_{\lambda,n}$, $(a_k,c)$ is just $\M+\A$ where $\M$ is
$\L_{\lambda,>c}$ or $\L_{\lambda,n,>c}$ depending on which type of order $a_k$
lives in and $(a_k,b)$ is just $\M+\K$.
In short, it is enough to demonstrate that
$$\A\geq_\lambda \K,$$
to show the claim.
To see this, choose $\alpha<\lambda$ and show that
$\A\geq_\alpha \K.$
Take $\delta_n>\alpha$. It is indeed the case that for $k\geq n$,
$\L_{\lambda}\equiv_\alpha \L_{\lambda,k}$. For some point $d$ in a block of
size $1$ between $c_n$ and $c_{n+1}$ we have that
\[\A=\A_{<d}+\sum_{q\in\Q} \A_q~~ \text{ and } ~~ \K=\A_{<d}+\sum_{q\in\Q} \B_q,\]
where $\A_q\equiv_\alpha \L_{\lambda}\equiv_\alpha \B_q$. This proves the claim.
\end{proof}

We need a different construction to deal with the $\Sinf{\lambda+3}$ case. 
However, it is arguably simpler than the construction we used for the example of
Scott sentence complexity $\Sinf{\lambda+2}$, as it is a sum of shuffle sums.
\begin{theorem}\label{thm:Slim+3}
For any limit ordinal $\lambda$, there is a linear order of Scott sentence
complexity $\Sinf{\lambda+3}.$
\end{theorem}

\begin{proof}
We let
$\L_1:=1+\Z+\omega^\lambda$ and $\L_2:=1+\Z+\omega^\lambda\cdot2$.
One important initial observation is that $\L_1\leq_{\lambda+1}\L_2$ as $\omega^\lambda\leq_{\lambda+1}\omega^\lambda\cdot2$.
As alluded to above, we take $\A=\L_1\cdot \Q$ and $\B=\L_2\cdot\Q$.
Finally we let $\L=\A+\L_2+\B$ and claim that $\L$ has the desired complexity
Let us begin with showing that $\SR(\A)\leq\lambda$. We define several helping predicates from which the claim will easily follow.
\begin{enumerate}
    \item Let $pt(x)$ if $x$ is in a block of size $1$. 
    \item Let $Z(x)$ say $\exists z<x ~pt(z)\land \forall y~z<y<x \to
        \bigvee_{n} S^n(y)=x.$ Note that $Z(x)$ holds if and only if $x$ is inside one of the copies of $\Z$.
    \item Let $init(x)$ if $\exists z<x ~ pt(z) \land (z,x)\cong\Z.$ Note that
        $init(x)$ if and only if $x$ is the first element in a copy of $\omega^\lambda$.
    \item Let $\gamma(x)$ if $\exists z<x ~init(z)\land [z,x)\cong\gamma$ for
        any $\gamma<\omega^\lambda$. Note that this definition is of complexity
        lower than $\Sinf\lambda$ as $\lambda$ is a limit ordinal. It describes
        the points that are of rank $\gamma$ within a copy of $\omega^{\lambda}$.
\end{enumerate}
Using these predicates one can define the automorphism orbits of any tuples in
$\A$ using formulas of complexity less than $\Sinf{\lambda+1}$.

All of these predicates have the same meaning within $\B$, however they leave
points in the second copy of $\omega^\lambda$ undefined. But it is still the case that $\SR(B)\leq\lambda+1$. To see this we need to define the additional relations
\begin{enumerate}
    \item Let 
        \[sec(x,z)\iff init(z) \land z<x \land \lnot\exists y~ \left(z<y<x\land
            pt(y)\right) \land \bigwedge_{\delta<\lambda}\forall y<x~ y\not\sim_\delta
        x.\]
        This $\Pinf\lambda$ predicate holds only if $z$ is the initial element in the second copy of $\omega^{\lambda}$ above the copy that $z$ is initial in.
    \item Let $\gamma_2(x)$ if $\exists y,z ~ sec(y,z) \land [y,x)\cong\gamma$ for any $\gamma\in\omega^\lambda$. Note that this definition is of complexity $\Sinf{\lambda+1}$ as $\lambda$ is a limit ordinal. It describes the points that are at the $\gamma$ position within a second copy of $\omega^{\lambda}$.
\end{enumerate}
Using these relations and those defined above, one can define the automorphism
orbits of tuples in $\B$.

In order to finish the proof of the claim, we appeal to Corollary \ref{lower} to
show that there is no $\Pinf{\lambda+3}$ Scott sentence.
This follows immediately as $\A\leq_{\lambda+1} \B$ is a direct consequence of $\L_1\leq_{\lambda+1}\L_2$.
\end{proof}

This leaves open the case of $\Sinf{\lambda+1}$. It is unclear to us if
such a linear ordering exists.

\begin{question}\label{ques:lambda+1}
    Is there a linear ordering of Scott sentence complexity $\Sinf{\lambda+1}$
    for $\lambda$ a limit ordinal?
\end{question}

All of the new examples of Scott sentence complexities given in this section
except \cref{prop:scatteredlimit} are linear orderings that contain $\Q$ as a
subordering and are thus non-scattered. While it is not hard to adapt the proof
of \cref{thm:nosigma3} to obtain that no scattered linear ordering can have
Scott sentence complexity $\Sinf{4}$, it seems difficult to look beyond that. An
analysis akin to the one presented in this article, but for Scott sentence
complexities of scattered linear orderings seems like it could prove interesting.
\begin{question}\label{ques:scatteredlo}
    Which Scott sentence complexities are realized by scattered linear
    orderings? In particular, for a given successor ordinal $\alpha>3$, is there a scattered linear
    ordering of Scott sentence complexity $\Sinf{\alpha}$?
\end{question}

\subsection{Remarks on effectivity.}
The examples we presented in this section are homogeneous enough so that if
$\alpha$ is a computable ordinal, then the examples witnessing Scott sentence
complexity $\Pinf{\alpha}$, $\Sinf{\alpha}$, or $\dSinf{\alpha}$ have computable
copies. Furthermore, these structures have c.e.\ Scott families, that is there
is a computable enumeration of codes for the formulas defining the automorphism
orbits of the tuples (potentially over a parameter). Thus, our examples give
computable linear orderings having computable Scott sentences.

The situation is a bit more tricky for examples of high Scott rank, that is
computable linear orderings with Scott rank $\ock$ or $\ock+1$. Our examples do
not give such structures. However, it is not hard to show that such structures
exist from existing examples in the literature and our results in
\cref{sec:fs-ssc}. We note that Calvert, Goncharov, and
Knight~\cite{calvert2007} already showed that there are computable linear
orderings of Scott rank $\ock$ and $\ock+1$. The following result is merely a
refinement.
\begin{theorem}\label{thm:highScottrank}
    There are computable linear orderings with Scott sentence complexities
    $\Pinf{\ock}$, $\Pinf{\ock+1}$, and $\Pinf{\ock+2}$.
\end{theorem}
\begin{proof}
    As is remarked in \cite{alvir2021}, the Harrison linear ordering $\ock\cdot (1+\Q)$ has Scott sentence
    complexity $\Pinf{\ock+2}$. The existence of structures of linear orderings
    of Scott sentence complexity $\Pinf{\ock+1}$ and $\Pinf{\ock}$ follows from
    \cref{thm:fsandscottrank} and the existence of a computable graph with these Scott
    sentence complexities.
\end{proof}
As for general limit ordinals we do not know whether there are computable linear orderings
of Scott sentence complexity $\Sinf{\ock+1}$. Additionally, we do not know
whether there is a computable linear of Scott rank $\dSinf{\ock+1}$.

\begin{question}\label{ques:highSinf}
    Is there a computable linear ordering of Scott sentence complexity
    $\Sinf{\ock+1}$?
\end{question}
\begin{question}\label{ques:highdSinf}
    Is there a computable linear ordering of Scott sentence complexity
    $\dSinf{\ock+1}$?
\end{question}
 
\printbibliography
\end{document}